\documentclass[a4paper, 11pt]{article}

\usepackage{a4wide}
\usepackage[english]{babel}
\usepackage{amsmath}
\usepackage{psfrag}
\usepackage[latin1]{inputenc}
\usepackage[T1]{fontenc}
\usepackage{amsthm}
\usepackage{amssymb}
\usepackage{verbatim}

\DeclareMathOperator{\res}{Res}
\DeclareMathOperator{\ind}{Ind}
\DeclareMathOperator{\Soc}{Soc}
\DeclareMathOperator{\Rad}{Rad}

\DeclareMathOperator{\End}{End}

\DeclareMathOperator{\im}{im}
\DeclareMathOperator{\Mat}{Mat}

\begin{document}
\newtheorem{defi}{Definition}[section]
\newtheorem{rem}[defi]{Remark}
\newtheorem{prop}[defi]{Proposition}
\newtheorem{ques}[defi]{Question}
\newtheorem{lemma}[defi]{Lemma}
\newtheorem{cor}[defi]{Corollary}
\newtheorem{thm}[defi]{Theorem}
\newtheorem{expl}[defi]{Example}
\renewcommand{\proofname}{\textsl{\textbf{Proof}}}

\begin{center}
{\bf\Large The vertices and sources of the natural simple module \\
\smallskip for the alternating group in even characteristic}

\bigskip
\textrm{Susanne Danz and J\"urgen M\"uller}

\medskip
\begin{abstract}
\noindent
For $n\geq 5$ the natural permutation module for the alternating
group $\mathfrak{A}_n$ has a unique non-trivial composition factor,
being called its natural simple module. We determine the vertices 
and sources of the natural simple $\mathfrak{A}_n$-module
over fields of characteristic $2$. 

\end{abstract}
\end{center}


\section{Introduction and result}
\normalfont

One of the leading themes in the representation
theory of finite groups is the question how far the
$p$-modular representation theory of a given group is determined
by local data, that is by its $p$-subgroups and their normalisers.
From this viewpoint, it is immediate to ask for the vertices 
and sources of the building blocks of all modules, namely the simple ones. 
Although not too much seems to be known in general, in recent years 
the picture has changed considerably for the symmetric group 
$\mathfrak{S}_n$ on $n$ letters, see for example the survey \cite{DK}.
Now it is natural to ask how the vertices and sources of the simple 
$\mathfrak{S}_n$-modules and those of the constituents of
their restrictions to the alternating group $\mathfrak{A}_n$ 
are related, and here in particular the natural $\mathfrak{S}_n$-module 
immediately springs to mind. The vertices and sources of
the natural $\mathfrak{S}_n$-module have been determined in \cite{MuZ},
and indeed the latter paper is the starting point
of the present work. As it turns out, the case of even
characteristic is by far the most interesting one.

\medskip
Hence, let $F$ be an algebraically closed field of characteristic $2$, 
let $n\geq 3$, and let $D:=D^{(n-1,1)}$ be the natural simple 
$F\mathfrak{S}_n$-module, that is, $D$ is the unique non-trivial 
composition factor of the 
natural permutation $F\mathfrak{S}_n$-module 
$M:=M^{(n-1,1)}:=\ind_{\mathfrak{S}_{n-1}}^{\mathfrak{S}_n}(F)$.
Now $\res_{\mathfrak{A}_n}^{\mathfrak{S}_n}(D)$ 
splits into two non-isomorphic simple modules 
$E^{(n-1,1)}_{\pm}$ if $n<5$, and otherwise 
$\res_{\mathfrak{A}_n}^{\mathfrak{S}_n}(D)=:E^{(n-1,1)}_0$ 
remains simple (cf. \cite{B}). 
In consequence of Mackey's Decomposition Theorem, 
we deduce that $E^{(n-1,1)}_0$ is the unique non-trivial composition 
factor of the natural permutation $F\mathfrak{A}_n$-module
$\ind_{\mathfrak{A}_{n-1}}^{\mathfrak{A}_n}(F)\cong
\res_{\mathfrak{A}_n}^{\mathfrak{S}_n}(M)$ if $n\geq 5$,
and we call $E^{(n-1,1)}_0$ the natural simple $F\mathfrak{A}_n$-module.
If $n<5$ then $\ind_{\mathfrak{A}_{n-1}}^{\mathfrak{A}_n}(F)$ has 
two non-trivial composition factors, $E^{(2,1)}_{\pm}$ and
$E^{(3,1)}_{\pm}$, respectively, also
called the natural simple $F\mathfrak{A}_3$-
and $F\mathfrak{A}_4$-modules, respectively.

\medskip
Our main aim now is to prove the following theorem.
Part of the assertions, listed for the sake of completeness,
are fairly immediate, and already covered in the subsequent remark.
There we also comment on the odd characteristic case,
and recall the results on the vertices and sources of the module $D$ 
from \cite{MuZ}.

\begin{thm}\label{thm:main}
Let $F$ be an algebraically closed field of characteristic $2$, 
and $n\geq 3$.
Let $E$ be the natural simple $F\mathfrak{A}_n$-module, and
let further $Q$ be a vertex of $E$.
Then the following hold:

\begin{enumerate}
\item[(i)]
If $n$ is odd then $Q$ is conjugate to a Sylow $2$-subgroup
of $\mathfrak{A}_{n-3}$, and $E$ has trivial source. 

\item[(ii)]
If $n>6$ is even then $Q$ is a Sylow $2$-subgroup of $\mathfrak{A}_n$, 
and $\res_Q^{\mathfrak{A}_n}(E)$ is a source of $E$.

\item[(iii)]
If $n=4$ then $Q$ is the Sylow $2$-subgroup of $\mathfrak{A}_4$,
and $E$ has trivial source. If $n=6$ then 
$Q$ is conjugate to
$(\mathfrak{S}_2\times\mathfrak{S}_2\times\mathfrak{S}_2)\cap\mathfrak{A}_6$,
and $E$ has a source of dimension $2$. 
\end{enumerate}
\end{thm}

\rem\label{rem:nat}\normalfont
(a) For the moment, we consider representations over an algebraically
closed field of odd characteristic $p$, and let $p\leq n$.
Then the natural simple module $D$ is not isomorphic to its conjugate
counterpart $D':=D\otimes\text{sgn}$, where $\text{sgn}$ denotes
the sign representation. 
Hence we have the natural simple $F\mathfrak{A}_n$-module 
$E:=\res_{\mathfrak{A}_n}^{\mathfrak{S}_n}(D)\cong 
    \res_{\mathfrak{A}_n}^{\mathfrak{S}_n}(D')$.
By \cite[Thm. 1.2, 1.3]{MuZ},
the vertices of $D$ coincide with the defect groups of its block,
which are contained in $\mathfrak{A}_n$ anyway and thus are the
vertices of $E$ as well. Moreover, $D$ and $E$ have the same sources,
which are trivial if $p\nmid n$, while if $p\mid n$ then are
the restrictions of $D$ to its vertices. This settles the case of odd 
characteristic, hence from now on we again stick to even characteristic.

\medskip
(b) In the case $n=3$, the simple $F\mathfrak{S}_3$-module 
$D^{(2,1)}$ as well as the simple $F\mathfrak{A}_3$-modules 
$E^{(2,1)}_{\pm}$ are projective and have thus trivial sources.
If $n>3$ is odd then $D^{(n-1,1)}\cong S^{(n-1,1)}$, 
where $S^{(n-1,1)}$ is the natural Specht $F\mathfrak{S}_n$-module,
and $E:=E^{(n-1,1)}_0=\res_{\mathfrak{A}_n}^{\mathfrak{S}_n}(D)$ is simple.
By \cite[Thm. 1.2, 1.3]{MuZ}, 
the vertices of $D$ are precisely the defect groups of
its block, which are conjugate to a Sylow $2$-subgroup of 
$\mathfrak{S}_{n-3}$, and $D$ has trivial source. 
Let $Q_{n-3}$ be a Sylow $2$-subgroup of $\mathfrak{A}_{n-3}$.
This is a defect group of the block of $F\mathfrak{A}_n$ containing $E$.
Since $D$ has trivial source, we also deduce that $F$
is a direct summand of $\res_{Q_{n-3}}^{\mathfrak{A}_n}(E)$. 
Therefore, also $E$ has the defect groups 
of its block as vertices and trivial source.

\medskip
(c) If $n$ is even then 
$D$ is contained in the principal block of $F\mathfrak{S}_n$. 
By \cite[Thm. 1.4(b), 1.5(b)]{MuZ}, 
$D^{(3,1)}$ has the Sylow 2-subgroup of $\mathfrak{A}_4$ as
vertex and trivial source.

\medskip
In \cite[Thm. 1.4(a), 1.5(a)]{MuZ} it is stated that 
$D^{(n-1,1)}$ has the Sylow 2-subgroup of $\mathfrak{S}_n$ 
as vertex, and its restriction to the vertex as source, for $n\geq 6$. 
This statement is correct, alone the proof of \cite[Prop. 4.5]{MuZ} 
used in between has a gap. It occurs in the very last line of 
\cite{MuZ}, and we have detected it while writing the present paper. 
Actually, we are able to use the techniques developed here 
to close that gap, and to give a new complete proof of 
\cite[Thm. 1.4(a), 1.5(a)]{MuZ}.  
Thus, in the present paper we only use 
\cite[Thm. 1.2, 1.3, 1.4(b), 1.5(b)]{MuZ}
as well as \cite[L. 3.3, Prop. 4.2]{MuZ},
which all are valid with the proofs given, at least as far as we see,
and we have made sure that we do not refer to any results which
have been proven elsewhere using \cite{MuZ}.

\rem\label{rem:strategy}\normalfont
Now the task is to determine the vertices of the natural simple
$F\mathfrak{A}_n$-module $E:=E^{(n-1,1)}_0$ whenever $n\geq 6$ is even. 
Hence from now on we suppose that $n\geq 4$ is even.
Suppose further that $Q_n$ is a Sylow $2$-subgroup
of $\mathfrak{A}_n$, and that $Q\leq Q_n$ is a 
vertex of $E$.

\medskip
(a) Then our general strategy is to show, firstly, 
that $\tilde{E}:=\res_{Q_n}^{\mathfrak{A}_n}(E)$ is indecomposable. 
This implies that $Q\leq Q_n$ also is a vertex of $\tilde{E}$.
Then, secondly, we assume that $\tilde{E}$ is
relatively projective with respect to some maximal subgroup $R<Q_n$
containing $Q$. In particular, we have $\Phi(Q_n)\leq R$,
where $\Phi(Q_n)$ denotes the Frattini subgroup of $Q_n$.
Moreover, we have
$\res_{\mathfrak{A}_{n-1}}^{\mathfrak{A}_n}(E)\cong E^{(n-2,1)}_0$,
where, by Remark \ref{rem:nat},
the latter has vertex $Q_{n-4}$ and trivial source $T$.
Thus $T$ is a direct summand of $\res_{Q_{n-4}}^{Q_n}(\tilde{E})$,
hence letting $S$ be a source of $\tilde{E}$
we infer that $T$ is a direct summand of
$\res_{Q_{n-4}}^{Q_n}(\ind_{Q}^{Q_n}(S))$.
Hence Mackey's Decomposition Theorem implies that there is $g\in Q_n$ such
that $Q^g\cap Q_{n-4}=Q_{n-4}$ (cf. \cite[L. 4.3.4]{NT}),
that is $Q_{n-4}\leq_{Q_n} Q\leq R<Q_n$.
Thus, since $R$ is normal in $Q_n$, we have $Q_{n-4}\leq R$ as well.
In summary we have $\langle\Phi(Q_n),Q_{n-4},Q\rangle\leq R$,
which typically turns out to be a fairly restrictive condition on $R$.

\medskip
Now, $\tilde{E}$ being relatively 
projective with respect to $R$, Green's Indecomposability 
Theorem implies that $\tilde{E}$ is induced from an 
indecomposable direct summand of $\res_R^{Q_n}(\tilde{E})$,
implying that $\res_R^{Q_n}(\tilde{E})$ is the direct sum
of two indecomposable modules of dimension $\dim(E)/2$ each.
This latter conclusion is then drawn to a contradiction,
implying that $\tilde{E}$, and thus $E$, cannot possibly
be relatively projective with respect to any proper subgroup of $Q_n$.

\medskip
(b) In the sequel, a key player will be the 
natural permutation $F\mathfrak{S}_n$-module $M$.
Letting $\{\gamma_1,\ldots,\gamma_n\}$ be its permutation basis,
$M$ is uniserial with composition series
$M> M'> M''> 0$, where
\begin{align*}
M'&:=\{\sum_{i=1}^na_i\gamma_i\in M\mid
       a_1,\ldots,a_n\in F,\,\sum_{i=1}^na_i=0\},\\
M''&:=\{a\sum_{i=1}^n\gamma_i\in M\mid a\in F\}\cong F\cong M/M',
\end{align*}
and $M'/M''\cong D$ as $F\mathfrak{S}_n$-modules.
In particular, $\dim(M')=n-1$ and
$\dim(D)=n-2$. Moreover,
$\{\gamma_i+\gamma_n\mid i\in\{1,\ldots,n-1\}\}$ is an $F$-basis of $M'$,
and if $^{\textbf{--}}:M\longrightarrow M/M''$
denotes the natural epimorphism then
$\{\overline{\gamma}_i+\overline{\gamma}_n\mid i\in\{1,\ldots,n-2\}\}$
is an $F$-basis of $M'/M''$ (cf. \cite[Ex. 5.1]{J}).
For convenience, in the following we will simply identify $D$ with $M'/M''$.

\medskip
This paper now is organised as follows: 
In Section \ref{sec:sbgrps} we begin by collecting some properties of the 
Sylow $2$-subgroups of $\mathfrak{S}_n$ and $\mathfrak{A}_n$, 
and their subgroups. 
It turns out that the behaviour of $E$ depends on the 
$2$-adic expansion of $n$. Hence letting $n=\sum_{j=1}^l 2^{i_j}$, 
for some $l\geq 1$ and $i_1>\ldots >i_l\geq 1$, we distinguish 
between several cases.
In Section \ref{sec:l1} we settle the case $l=1$, that is, 
$n$ is a $2$-power.
In Section \ref{sec:l2} we begin to investigate the case $l\geq 2$,
that is, $n$ is even but not a $2$-power, by first considering 
restrictions of $E$ to various abelian subgroups.
This leads to further case distinctions with respect to $n_l:=2^{i_l}$.
In Section \ref{sec:nlgt2} we settle the case $n_l>2$, 
while for $n_l=2$ in Sections \ref{sec:nl2lgeq4}--
\ref{sec:nl2l2} the cases $l\geq 4$, $l=3$, and $l=2$, 
are dealt with, respectively. In the very end the case $n=6$ 
remains to be considered, completing the proof of Theorem \ref{thm:main}.
In the final Section \ref{sec:appendix}, which is an appendix,
we in particular use the technique employed in Section \ref{sec:nl2l2} 
to give the new proof of \cite[Thm. 1.4(a), 1.5(a)]{MuZ}.

\medskip
We want to point out that, although no explicit reference
to computer calculations is made in the present paper,
before writing it we have dealt with various examples 
by way of explicit computations, which helped us to understand
the different behaviour of $E$ with respect to the $2$-adic
expansion of $n$.
To do so, we have used the computer algebra systems
{\sf GAP} \cite{GAP} and {\sf MAGMA} \cite{MAGMA},
and the specially tailored computational techniques 
to determine vertices and sources developed in \cite{DKZ}. 
Finally, for any finite group $G$, an
$FG$-module is understood to be a finitely generated right $FG$-module.
Furthermore, the endomorphism algebra $\End_{FG}(M)$ is also supposed 
to act from the right. 
For an introduction to the theory of vertices and sources we refer to 
\cite[Sec. 4.3]{NT}, and for details concerning the
representation theory of the symmetric groups to \cite{J} and
\cite{JK}.

\medskip
{\bf Acknowledgement.} The first author's research on
this article has been supported by the Deutsche Forschungsgemeinschaft
through DFG grant \#DA 1115/1-1. 
The second author gratefully acknowledges financial support and
the enjoyable hospitality of Friedrich-Schiller-Universit\"at Jena,
where parts of the present paper have been written.


\section{Sylow subgroups of the symmetric and alternating groups}
\label{sec:sbgrps}
\normalfont

We begin by introducing our notation for the Sylow $2$-subgroups of 
$\mathfrak{S}_n$ and $\mathfrak{A}_n$, respectively.
Moreover, we collect some properties of these Sylow subgroups and 
their subgroups which will be needed in the course of the 
subsequent sections. 

\begin{rem}\label{rem:sylow}
\normalfont
(a) First of all, let $n=2^m$, for some $m\in\mathbb{N}_0$. 
We set $P_1:=1$, $P_2:=C_2:=\langle (1,2)\rangle$, and
$P_{2^{i+1}}:=P_{2^i}\wr C_2
=\{(x_1,x_2;\sigma)\mid x_1,x_2\in P_{2^i},\, \sigma\in P_2\}$, 
for $i\geq 1$. Recall that, for $i\geq 1$ and
$(x_1,x_2;\sigma),\, (y_1,y_2;\pi)\in P_{2^{i+1}}$, we have 
$(x_1,x_2;\sigma)(y_1,y_2;\pi)
=(x_1y_{1^\sigma},x_2y_{2^\sigma};\sigma\pi)$.
As usual, for $i\geq 0$, we regard $P_{2^i}$ as a subgroup of 
$\mathfrak{S}_{2^i}$ in the obvious way. 
Then, by \cite[4.1.22, 4.1.24]{JK}, 
$P_n=P_{2^m}$ is a Sylow $2$-subgroup of $\mathfrak{S}_{2^m}=\mathfrak{S}_n$.

\medskip
Moreover, by construction, $P_{2^m}$ is generated by the following elements:
\begin{equation}\label{equ:Pgens}
w_{2^i}:=\prod_{k=1}^{2^{i-1}}(k,k+2^{i-1}),
\end{equation}
for $i=1,\ldots,m$. 
For instance, $P_8$ is generated by 
$w_2=(1,2)$, $w_4=(1,3)(2,4)$ and $w_8=(1,5)(2,6)(3,7)(4,8)$. 
Additionally let $w_1:=1$. Since, by \cite[Satz III.15.3]{H}, 
we have $|P_{2^m}:\Phi(P_{2^m})|=2^m$, (\ref{equ:Pgens}) 
yields in fact a minimal set of generators for $P_n=P_{2^m}$.

\medskip
If $m\geq 2$, then we have
$P_n=P_{2^m}=P_{2^{m-1}}\wr P_2
=\{(x_1,x_2;\sigma)\mid x_1,x_2\in P_{2^{m-1}},\, \sigma\in P_2\}$. 
The base group of $P_{2^{m-1}}\wr P_2$ shall henceforth be denoted by $B_n$. 
When viewing $P_{2^m}$ as a subgroup of $\mathfrak{S}_{2^m}$ as above, 
$B_n$ corresponds to $P_{2^{m-1}}\times w_nP_{2^{m-1}}w_n$.

\medskip
(b) Let now $n\geq 2$ be even with $2$-adic expansion 
$n=\sum_{j=1}^l2^{i_j}$, for appropriate $l\geq 2$ and 
$i_1>\ldots >i_l\geq 1$, and let $n_j:=2^{i_j}$ for $j=1,\ldots,l$.
Then, by \cite[4.1.22, 4.1.24]{JK},
$P_n=\prod_{j=1}^l P_{n_j}^{g_j}$,
with $g_1=1$, and 
$g_j:=\prod_{k=1}^{n_j}(k,k+\sum_{s=1}^{j-1}n_s)$ for $j=2,\ldots,l$,
is a Sylow $2$-subgroup of $\mathfrak{S}_n$. 
For convenience, we will simply write
$P_n=\prod_{j=1}^lP_{n_j}$ where $P_{n_j}$ 
is understood to be acting on the subset 
$\Omega_j:=\{1+\sum_{s=1}^{j-1}n_s,\ldots,\sum_{s=1}^j n_s\}$ 
of $\{1,\ldots,n\}$, for $j=1,\ldots,l$.
If $n$ is odd then we simply set $P_n:=P_{n-1}$,
so that again $P_n$ is a Sylow $2$-subgroup of $\mathfrak{S}_n$.

\medskip
Via (\ref{equ:Pgens}) we obtain
a minimal set of generators for $P_{n_j}$ 
denoted by $\mathcal{W}_j:=\{w_{2,j},\ldots,w_{n_j,j}\}$, 
for $j=1,\ldots,l$. That is, 
$\mathcal{W}:=\bigcup_{j=1}^l\mathcal{W}_j$
is a minimal generating set for $P_n$. 
For instance, $P_{14}=P_8\times P_4\times P_2$ is generated by
$w_{2,1}=(1,2)$, 
$w_{4,1}=(1,3)(2,4)$, 
$w_{8,1}=(1,5)(2,6)(3,7)(4,8)$, 
$w_{2,2}=(9,10)$, $w_{4,2}=(9,11)(10,12)$ and
$w_{2,3}=(13,14)$.

\medskip
(c) Now, for any $n\geq 2$, we set 
$Q_n:=P_n\cap\mathfrak{A}_n$ so that $Q_n$ is a Sylow $2$-subgroup of 
$\mathfrak{A}_n$. 

\medskip
If $n=2^m$, for some $m\geq 2$, then writing $P_n=P_{2^{m-1}}\wr P_2$
we have 
$Q_n=\{(x_1,x_2;\sigma)\in P_{2^{m-1}}\wr P_2 \mid 
  x_1x_2\in Q_{2^{m-1}}\}\leq P_{2^{m-1}}\wr P_2$,
and we define
\begin{equation}\label{equ:Qgens}
w_2':=w_2\cdot w_2^{w_{2^m}}=(1,2)(2^{m-1}+1,2^{m-1}+2),
\text{ and } w_{2^i}':=w_{2^i}\text{ for } i=2,\ldots,m.
\end{equation}

\medskip
If $n=\sum_{j=1}^l2^{i_j}=\sum_{j=1}^l n_j$, for some $l\geq 2$ and some 
$i_1>\ldots >i_l\geq 1$, is the $2$-adic expansion
of $n$ as above then we define
\begin{equation}\label{equ:Qgens2}
w_{2,j}':=w_{2,j}w_{2,l} \text{ for } j=1,\ldots,l-1
\text{, and } w_{2^s,j}':=w_{2^s,j} 
\text{ for } j=1,\ldots, l \text{ and } s=2,\ldots,i_j.
\end{equation}
\end{rem}

\begin{prop}\label{prop:mingens}
With the notation of the previous remark we have:

\begin{enumerate}
\item[(i)] If $n=2^m$, for some $m\geq 2$, then 
$Q_n=\langle w_2',\ldots,w_{2^m}'\rangle$. 
Furthermore, $\{w_2',\ldots, w_{2^m}'\}$
is a minimal set of generators for $Q_n$.
\item[(ii)] If $n=\sum_{j=1}^l2^{i_j}=\sum_{j=1}^l n_j$, for some $l\geq 2$ 
and some $i_1>\ldots >i_l\geq 1$, then
$$\mathcal{W}':=\bigcup_{j=1}^{l-1}\{w_{2,j}',\ldots,
   w_{n_j,j}'\}\cup\{w_{4,l}',\ldots,w_{n_l,l}'\}$$ 
is a minimal generating set for $Q_n$.
\end{enumerate}
\end{prop}

\begin{proof}
Suppose first that $n=2^m$, for some $m\geq 2$. 
Obviously, $Q:=\langle w_2',\ldots,w_{2^m}'\rangle\leq Q_n$.
Moreover, we have $w_2w_{2^i}'w_2=w_2'w_{2^i}'w_2'\in Q$, for 
$i=1,\ldots,m-1$, and
$$w_2w_{2^m}'w_2= w_2w_{2^m}w_2\cdot w_{2^m}w_{2^m}=w_2'w_{2^m}'\in Q  ,$$
so that $w_2Qw_2=Q$. This shows that 
$\langle w_2\rangle Q=Q\langle w_2\rangle=P_n$,
and thus $|P_n:Q|=2$, that is $Q=Q_n$. 
It remains to show that $\{w_2',\ldots,w_{2^m}'\}$ 
is a minimal generating set for $Q_n$.
For this, notice that $Q_n$ acts transitively on $\{1,\ldots,n\}$, 
whereas for any $j\in\{1,\ldots,m\}$ the group
$Q:=\langle w_{2^i}'\mid i\in\{1,\ldots,m\}\setminus\{j\}\rangle$ does not,
namely, 1 and $1+2^{j-1}$ then lie in different $Q$-orbits. 
Consequently, $\{w_2',\ldots,w_{2^m}'\}$ is in fact minimal.

\medskip
Now suppose that $n=\sum_{j=1}^l2^{i_j}=\sum_{j=1}^l n_j$, 
for some $l\geq 2$ and some $i_1>\ldots >i_l\geq 1$. 
Then $Q:=\langle \mathcal{W}'\rangle\leq Q_n$.
Furthermore, we have
\begin{align*}
w_{2,l}w_{2^s,j}'w_{2,l}&= w_{2^s,j}'\in Q\text{, for } 
j=1,\ldots,l-1,\; s=1,\ldots,i_j,\\
w_{2,l}w_{2^s,l}'w_{2,l}&=w_{2,1}'w_{2^s,l}'w_{2,1}'\in Q\text{, for } 
s=2,\ldots,i_l,
\end{align*}
hence $w_{2,l}Qw_{2,l}=Q$, and clearly 
$\langle w_{2,l}\rangle Q=Q\langle w_{2,l}\rangle=P_n$. 
This shows that $|P_n:Q|=2$ so that $Q=Q_n$. Furthermore, 
from Remark \ref{rem:sylow} we deduce that 
$\mathcal{W}'\cup\{w_{2,l}\}$ is a minimal generating set for $P_n$, 
and thus also $\mathcal{W}'$ has to be a minimal generating set for $Q_n$. 
\end{proof}

\begin{expl}\label{expl:mingens}
\normalfont
In order to illustrate the rather technical notation above, 
consider the cases where $n=8$ and $n=14=8+4+2$, respectively:
Then $Q_8$ is generated by 
$$ w_2'=(1,2)(5,6), \quad 
   w_4'=(1,3)(2,4), \quad
   w_8'=(1,5)(2,6)(3,7)(4,8).$$
Analogously, $Q_{14}$ is generated by 
\begin{align*}
w_{2,1}'&=(1,2)(13,14), &
w_{4,1}'&=(1,3)(2,4), & 
w_{8,1}'&=(1,5)(2,6)(3,7)(4,8), \\
w_{2,2}'&=(9,10)(13,14), &
w_{4,2}'&=(9,11)(10,12). 
\end{align*}
\end{expl}

\begin{prop}\label{prop:frattini}
For any $n\geq 2$ we have $\Phi(P_n)=[P_n,P_n]$ and 
$\Phi(Q_n)=[Q_n,Q_n]$. If $n=2^m$, for some $m\geq 2$,
then $|\Phi(P_n):\Phi(Q_n)|=2$, otherwise $\Phi(P_n)=\Phi(Q_n)$.
\end{prop}

\begin{proof}
By \cite[Satz III.3.14]{H}, $[P_n,P_n]\leq \Phi(P_n)$ and 
$[Q_n,Q_n]\leq \Phi(Q_n)$.
As we have just seen, both $P_n$ and $Q_n$ are generated by 
elements of order $2$. Consequently, the same holds true for
$P_n/[P_n,P_n]$ and $Q_n/[Q_n,Q_n]$. Thus both $P_n/[P_n,P_n]$ 
and $Q_n/[Q_n,Q_n]$ are elementary abelian. 
This yields $[P_n,P_n]=\Phi(P_n)$ and 
$[Q_n,Q_n]= \Phi(Q_n)$. 
The second assertion is immediate from 
Proposition \ref{prop:mingens} and Remark \ref{rem:sylow}.
\end{proof}

\begin{prop}\label{prop:frattiniQ}
Let $n=2^m$, for some $m\geq 2$. Then we have
$\Phi(P_n)=\{(x_1,x_2;1)\mid x_1x_2\in\Phi(P_{2^{m-1}})\}$ and
$\Phi(Q_n)=\{(x_1,x_2;1)\mid 
 x_1,x_2\in Q_{2^{m-1}}, x_1x_2\in\Phi(P_{2^{m-1}})\}$.
\end{prop}

\begin{proof}
The first assertion follows from \cite[L. 1.4]{O} 
and the fact that $\Phi(P_n)=[P_n,P_n]$. It remains to show that
$\Phi(Q_n)=H:=\{(x_1,x_2;1)\mid x_1,x_2
 \in Q_{2^{m-1}}, x_1x_2\in\Phi(P_{2^{m-1}})\}\leq \Phi(P_n)$.
For this let $x=(x_1,x_2;\sigma)\in Q_n$. Then 
$x_1,x_2\in P_{2^{m-1}}$, $x_1x_2,x_2x_1\in Q_{2^{m-1}}$ 
and $\sigma\in\{(1,2),1\}$. Moreover,
\begin{equation*}
x^2=\begin{cases}
(x_1^2,x_2^2;1), &\text{ if } \sigma=1,\\
(x_1x_2,x_2x_1;1),&\text{ if }\sigma\neq 1.
\end{cases}
\end{equation*}
That is, $x^2=(y_1,y_2;1)$ where $y_1,y_2\in Q_{2^{m-1}}$, and 
$x^2\in\Phi(P_n)$, by \cite[Satz III.3.14]{H}. Thus 
$y_1y_2\in\Phi(P_{2^{m-1}})$.
Since, again by \cite[Satz III.3.14]{H}, 
$\Phi(Q_n)$ is generated by the squares in $Q_n$, this implies 
$\Phi(Q_n)\leq H\leq \Phi(P_n)$. But $Q_{2^{m-1}}< P_{2^{m-1}}$ 
so that $H\neq \Phi(P_n)$. Therefore, from Proposition \ref{prop:frattini}
we deduce that $H=\Phi(Q_n)$.
\end{proof}

\begin{prop}\label{prop:basis}
Suppose that $n=2^m$, for some $m\geq 2$. Then the following hold:
\begin{enumerate}
\item[(i)] 
We have
$B_n=\langle w_2,\ldots, w_{2^{m-1}}, w_2^{w_n},\ldots,
             w_{2^{m-1}}^{w_n}\rangle$
and
$$ B_n\cap Q_n=\langle w_2w_2^{w_n}, w_4,\ldots,w_{2^{m-1}},w_4^{w_n},
                       \ldots,w_{2^{m-1}}^{w_n}\rangle .$$
\item[(ii)] 
We have $P_{2^{m-1}}\Phi(P_n)=B_n$ and $Q_{2^{m-1}}\Phi(P_n)=B_n\cap Q_n$.
\item[(iii)] If $m\geq 3$ then we have $Q_{2^{m-1}+2}\Phi(Q_n)=B_n\cap Q_n$.
\end{enumerate}
\end{prop}

\begin{proof}
Since $B_n=P_{2^{m-1}}\times P_{2^{m-1}}^{w_n}$, 
the first statement in (i) is clear from Remark \ref{rem:sylow}. Furthermore,
$\langle w_2w_2^{w_n}, w_4,\ldots,w_{2^{m-1}},w_4^{w_n},\ldots,
 w_{2^{m-1}}^{w_n}\rangle\leq B_n\cap\mathfrak{A}_n=B_n\cap Q_n$,
and is normalized by $w_2$. Since
$\langle w_2\rangle\cdot\langle w_2w_2^{w_n},w_4,
 \ldots,w_{2^{m-1}},w_4^{w_n},\ldots,w_{2^{m-1}}^{w_n}\rangle=B_n$,
also the second assertion in (i) follows.

\medskip
Now we show that $B_n\leq P_{2^{m-1}}\Phi(P_n)$. For this let 
$x\in B_n$, and write $x=(x_1,x_2;1)$ with $x_1,x_2\in P_{2^{m-1}}$.
Then $x=(x_1x_2,1;1)(x_2^{-1},x_2;1)$. Since $x_1x_2\in P_{2^{m-1}}$ and 
$(x_2^{-1},x_2;1)\in\Phi(P_n)$, we get $x\in P_{2^{m-1}}\Phi(P_n)$.
Consequently, $P_{2^{m-1}}\Phi(P_n)$ is a proper subgroup of $P_n$ 
containing the maximal subgroup $B_n$ of $P_n$, and hence
$B_n=P_{2^{m-1}}\Phi(P_n)$.

\medskip
Next we verify that $B_n\cap Q_n=Q_{2^{m-1}}\Phi(P_n)$. 
By (i), we already know that 
$Q_{2^{m-1}}\Phi(P_n)\leq (P_{2^{m-1}}\Phi(P_n))\cap Q_n=B_n\cap Q_n$.
Now let $x\in B_n\cap Q_n$, and again write $x=(x_1,x_2;1)$ for
appropriate $x_1,x_2\in P_{2^{m-1}}$. Since $x\in Q_n$, 
we have $x_1x_2\in Q_{2^{m-1}}$, and thus 
$x=(x_1x_2,1;1)(x_2^{-1},x_2;1)\in Q_{2^{m-1}}\Phi(P_n)$.
Hence $Q_{2^{m-1}}\Phi(P_n)=B_n\cap Q_n$, proving (ii).

\medskip
Now let $m\geq 3$, then
$Q_{2^{m-1}+2}\Phi(Q_n)\leq (P_{2^{m-1}+2}\Phi(P_n))\cap Q_n=B_n\cap Q_n$.
The last equation follows from the fact that $P_{2^{m-1}+2}<P_n$ so that 
$P_{2^{m-1}+2}\Phi(P_n)$ is a proper subgroup of $P_n$
containing the maximal subgroup $B_n=P_{2^{m-1}}\Phi(P_n)$. Since 
$$\Phi(Q_n)\cap Q_{2^{m-1}+2}=\{(x_1,1;1)\mid 
  x_1\in\Phi(P_{2^{m-1}})\}=Q_{2^{m-1}}\cap\Phi(P_n),$$ 
we now have
\begin{align*}
|Q_{2^{m-1}+2}\Phi(Q_n)|&
=\frac{|Q_{2^{m-1}+2}||\Phi(Q_n)|}{|\Phi(Q_n)\cap Q_{2^{m-1}+2}|}
=\frac{|Q_{2^{m-1}+2}||\Phi(Q_n)|}{|\Phi(P_{2^{m-1}})|}
=\frac{|Q_{2^{m-1}}|\cdot 2\cdot |\Phi(Q_n)|}{|\Phi(P_{2^{m-1}})|}\\
&=\frac{|Q_{2^{m-1}}||\Phi(P_n)|}{|\Phi(P_{2^{m-1}})|}
=|Q_{2^{m-1}}\Phi(P_n)|=|B_n\cap Q_n|.
\end{align*}
This finally shows that also $Q_{2^{m-1}+2}\Phi(Q_n)=B_n\cap Q_n$, 
and assertion (iii) follows.
\end{proof}

\begin{rem}\label{rem:xy}
\normalfont
(a) Again let $n=\sum_{j=1}^l2^{i_j}=\sum_{j=1}^ln_j$, for some
$l\geq 1$ and some $i_1>\ldots >i_l\geq 1$, be
the $2$-adic expansion of $n$.
For $j=1,\ldots,l$ define $y_{n_j}:=w_{n_j,j}\cdots w_{2,j}$,
where the $w_{2,j},\ldots,w_{n_j,j}$ are as in Remark \ref{rem:sylow}. 
Then, by \cite[L. 3.3]{MuZ}, $y_{n_j}$ is an $n_j$-cycle, 
for $j=1,\ldots,l$. Furthermore, 
let $Y_{n_j}:=\langle y_{n_j}\rangle$, for $j=1,\ldots,l$, 
and $Y_n:=\prod_{j=1}^lY_{n_j}$; clearly $Y_n\nleq \mathfrak{A}_n$.
Next we set $y_{n_j}':=y_{n_j}y_{n_l}$ and 
$Y_{n_j}':=\langle y_{n_j}'\rangle$.
Then we have
$Y_n':=Y_n\cap\mathfrak{A}_n=\prod_{j=1}^lY_{n_j}'$,
namely, we clearly have 
$Y:=\langle  y_{n_1}',\ldots,y_{n_l}'\rangle\leq Y_n'$,
and from $\langle y_{n_l}\rangle Y=Y\langle y_{n_l}\rangle=Y_n$
where $y_{n_l}\notin Y$ but $y_{n_l}^2\in Y$ we get $|Y_n:Y|=2$,
hence $Y=Y_n'$.

\medskip
Moreover, we set
$x_{n_j}:=y_{n_j}^2$, for $j=1,\ldots,l$, that is, 
$x_{n_j}$ has cycle type $(n_j/2,n_j/2)$, for $j=1,\ldots,l$.
We also set $X_{n_j}:=\langle x_{n_j}\rangle$, for $j=1,\ldots,l$, 
and $X_n:=\prod_{j=1}^l X_{n_j}$.
Note that, if $l\geq 2$ then the $x_{n_j}$ are squares in
$Q_n$, so that $X_n\leq\Phi(Q_n)$.

\medskip
(b) When viewing the wreath product $P_{n_j}=P_{\frac{n_j}{2}}\wr C_2$
as a subgroup of $\mathfrak{S}_{n_j}$ as usual, 
the subgroup of $P_{n_j}$ 
corresponding to the base group is isomorphic to 
$P_{\frac{n_j}{2}}\times P_{\frac{n_j}{2}}$ and shall be denoted by 
$B_{n_j}:=P_{n_j}'\times P_{n_j}''$, for $j=1,\ldots,l$. Here 
$P_{n_j}'\cong P_{\frac{n_j}{2}}$ is supposed to be acting on 
$\Omega_j':=\{n_1+\cdots +n_{j-1}+1,\ldots, n_1+\cdots +n_{j-1}+n_j/2\}$, 
and $P_{n_j}''\cong P_{\frac{n_j}{2}}$ acting on 
$\Omega_j'':=\{n_1+\cdots +n_{j-1}+n_j/2+1,\ldots, n_1+\cdots +n_j\}$.
We also set $B_n:=\prod_{j=1}^lB_{n_j}$, as well as
$B_{n_j}':=B_{n_j}\cap \mathfrak{A}_n\leq Q_{n_j}$, for $j=1,\ldots,l$, and 
$B_n':=B_n\cap\mathfrak{A}_n\leq Q_n$. Notice that 
$y_{n_j}^2=w_{n_j,j}w_{n_{j-1},j}\cdots 
 w_{2,j}w_{n_j,j}\cdot w_{n_{j-1},j}\cdots w_{2,j}\in B_{n_j}'$, 
for $j=1,\ldots,l$. 
\end{rem}

%


\section{The case $l=1$}\label{sec:l1}
\normalfont

We investigate the case where $n\geq 8$ is a $2$-power. 

\begin{lemma}\label{lemma:possmax}
Let $n=2^m$, for some $m\geq 3$. 
Suppose that $R$ is a maximal subgroup of $Q_n$ such that $E$ is 
relatively $R$-projective. Then, if $m>3$ we have $R=B_n\cap Q_n=B_n'$.
If $m=3$ then we have
$R\in\{B_8', (Q_4\times Q_4)\langle w_8\rangle, 
             (Q_4\times Q_4)\langle w_8^{w_2}\rangle\}$.  
\end{lemma}

\begin{proof}
By \cite[Prop. 4.2]{MuZ}, 
we know that $\res_{Q_n}^{\mathfrak{A}_n}(E)$ 
is indecomposable. Hence fix some vertex $Q<Q_n$ of $E$,
such that $Q\leq R$. Then, by Remark \ref{rem:strategy}, we have 
$\langle\Phi(Q_n),Q_{n-4},Q\rangle\leq R$.
If $n=8$ then $Q_4\Phi(Q_8)$ is a normal subgroup of $Q_8$, 
and $Q_8/Q_4\Phi(Q_8)$ is elementary abelian of order 4. 
Thus there are three maximal subgroups of $Q_8$ containing 
$Q_4\Phi(Q_8)$, and these are precisely the ones listed;
note that $B_8'=(Q_4\times Q_4)\langle w_2'\rangle$
where $w_2'=w_8^{w_2}\cdot w_8$.
Hence we may now suppose that $m>3$, and it suffices to prove
$\Phi(Q_n)Q_{n-4}=Q_{n-4}\Phi(Q_n)=B_n\cap Q_n$.
But, since $m>3$, we have $n-4=2^m-4> 2^{m-1}+2$, 
and the latter assertion
is immediate from Proposition \ref{prop:basis} and the fact
that $Q_{n-4}<Q_n$.
\end{proof}

\begin{prop}\label{prop:vertex1}
Let $n=2^m$, for some $m\geq 3$. Then $E$ has vertex $Q_n$ 
and source $\res_{Q_n}^{\mathfrak{A}_n}(E)$.
\end{prop}

\begin{proof}
As already mentioned, $\res_{Q_n}^{\mathfrak{A}_n}(E)$ is indecomposable,
by \cite[Prop. 4.2]{MuZ}. 
We follow the strategy given in Remark \ref{rem:strategy},
and assume that $R<Q_n$ is a maximal subgroup such that $E$ is relatively
$R$-projective.

\medskip
Let first $m>3$. Then, by Lemma \ref{lemma:possmax}, we have
$R=B_n\cap Q_n\geq\Phi(P_n)$.
In particular, $\langle x_n\rangle=X_n\leq \Phi(P_n)\leq R$ 
where $x_n$ is the permutation of cycle type $(n/2,n/2)$ 
defined in Remark \ref{rem:xy}. 
As in the proof of \cite[Prop. 4.2]{MuZ} we get
$\res_{X_n}^{\mathfrak{A}_n}(E)\cong 
 T_{\frac{n-2}{2}}\oplus T_{\frac{n-2}{2}}$,
where $T_{\frac{n-2}{2}}$ denotes the indecomposable $FX_n$-module of 
dimension $(n-2)/2$. 
Thus we conclude that $\res_R^{\mathfrak{A}_n}(E)=E_1\oplus E_2$
where $E_1$ and $E_2$ are in fact uniserial with trivial
heads and trivial socles. Thus, in particular, 
$\Soc(\res_{X_n}^{\mathfrak{A}_n}(E))
=\res_{X_n}^R(\Soc(\res_R^{\mathfrak{A}_n}(E)))$.
We set 
$$s_1:=\sum_{i=1}^{n/2}\gamma_i\in M'
\quad\text{ and }\quad 
s_2:=\sum_{i=1}^{n/4}\gamma_i+\sum_{i=1+n/2}^{3n/4}\gamma_i\in M'.$$
Then $s_1(x_n+1)=0$ and $s_2(x_n+1)=\sum_{i=1}^n\gamma_i\in M''$. 
Consequently, $\overline{s}_1:=s_1+M''$ and $\overline{s}_2:=s_2+M''$ are
annihilated by $\Rad(FX_n)$, that is they are contained in 
$\Soc(\res_{X_n}^{\mathfrak{A}_n}(E))$. 
Furthermore, $\overline{s}_1$ and $\overline{s}_2$ 
are linearly independent over $F$. 
Therefore $\{\overline{s}_1,\overline{s}_2\}$ is an $F$-basis for
$\Soc(\res_R^{\mathfrak{A}_n}(E))$. 
But we also have $w_{2^{m-1}}\in B_n\cap Q_n=R$, 
yielding the contradiction 
$$\overline{s}_2w_{2^{m-1}}
=\sum_{i=1+n/4}^{3n/4}\overline{\gamma}_i\neq\overline{s}_2.$$

\medskip
Let finally $m=3$. Then $R$ is one of the subgroups given in 
Lemma \ref{lemma:possmax}.  
To exclude the case $R_1:=B_8'$, notice that, by Remark \ref{rem:xy},
we have $x_8=y_8^2=(1,3,2,4)(5,7,6,8)\in R_1$,
and hence we may argue as in the case $m>3$. 
As for $R_2:=(Q_4\times Q_4)\langle w_8\rangle$, 
we have $w_8w_4=(1,5,3,7)(2,6,4,8)\in R_2$,
and thus letting $s_1:=\gamma_1+\gamma_3+\gamma_5+\gamma_7\in M'$
and $s_2:=\gamma_1+\gamma_3+\gamma_2+\gamma_4\in M'$ instead
we similarly arrive at a contradiction, using
$s_1 w_4^{w_2}=\gamma_2+\gamma_4+\gamma_5+\gamma_7$
where $w_4^{w_2}=(1,4)(2,3)\in Q_4$.
For $R_3:=(Q_4\times Q_4)\langle w_8^{w_2}\rangle$ 
again argue similarly, using
$w_8^{w_2}w_4=(1,6,3,7)(2,5,4,8)\in R_3$,
and $s_1:=\gamma_1+\gamma_3+\gamma_6+\gamma_7\in M'$
and $s_2:=\gamma_1+\gamma_3+\gamma_2+\gamma_4\in M'$,
where $s_1 w_4^{w_2}=\gamma_2+\gamma_4+\gamma_6+\gamma_7$.
\end{proof}

\medskip
This completes the proof of Theorem \ref{thm:main} in the case
where $n$ is a $2$-power.


\section{The case $l\geq 2$}\label{sec:l2}
\normalfont

We now investigate the case where $n\geq 4$ is even, but not a $2$-power.
We aim to show that the vertices of 
$E=\res_{\mathfrak{A}_n}^{\mathfrak{S}_n}(D)$ are then always 
the Sylow $2$-subgroups of $\mathfrak{A}_n$, unless $n=6$. 
We first need a few preparations, concerning restrictions of $E$
to various subgroups:
Recall from Remark \ref{rem:xy} that $Q_n$ contains the abelian subgroups
$Y_n'$ and $X_n$, where, since $l\geq 2$, we moreover have $X_n\leq\Phi(Q_n)$.
We will investigate the restrictions of $E$ to $Y_n'$ and $X_n$, respectively,
in detail. Thus we gain information on restrictions of
$E$ to $Q_n$ and to maximal subgroups of $Q_n$. 
This leads to a further division into various subcases, 
which will be dealt with in the subsequent sections.

\begin{rem}\label{rem:ybas}
\normalfont
In Remark \ref{rem:nat} we introduced the permutation basis 
$\{\gamma_1,\ldots,\gamma_n\}$ of the natural permutation 
$F\mathfrak{S}_n$-module $M$. While working with the group $Y_n'$, 
it will be convenient to re-number the elements of this
basis as follows: 
for $j=1,\ldots,l$ and $i=n_1+\cdots + n_{j-1}+1,\ldots, n_1+\cdots +n_j$, 
we set
$$\delta_i:=\gamma_{n_1+\cdots +n_{j-1}+1}y_{n_j}^{i-n_1-\cdots -n_{j-1}-1}.$$
In particular, for $j=1,\ldots,l$ and 
$i=n_1+\cdots +n_{j-1}+1,\ldots,n_1+\cdots + n_j-1$ we have
$\delta_iy_{n_j}=\delta_{i+1},$
and $\delta_{n_1+\cdots +n_j}y_{n_j}=\delta_{n_1+\cdots +n_{j-1}+1}$. 
That is, $\{\delta_{n_1+\cdots +n_{j-1}+1},\ldots,\delta_{n_1+\cdots +n_j}\}$ 
is the permutation basis of the natural 
permutation $FY_{n_j}$-module, for $j=1,\ldots,l$. 

\medskip
For $j=1,\ldots,l$ we now define
\begin{align*}
\delta_{j'}^+&:=\sum_{i=0}^{n_j/2-1}\delta_{n_1+\cdots +n_{j-1}+1+2i}
=\sum_{i\in\Omega_j'}\gamma_i=:\gamma_{j'}^+,\\
\delta_{j''}^+&:=\sum_{i=0}^{n_j/2-1}\delta_{n_1+\cdots +n_{j-1}+2+2i}
=\sum_{i\in\Omega_j''}\gamma_i=:\gamma_{j''}^+,
\end{align*}
and
$$ \delta_0^+:=\sum_{i=0}^{n/2-1}\delta_{2+2i}=\sum_{j=1}^l\delta_{j''}^+
  =\sum_{j=1}^l\gamma_{j''}^+=:\gamma_0^+,$$
as well as 
$\delta_j^+:=\delta_{j'}^++\delta_{j''}^+
=\gamma_{j'}^++\gamma_{j''}^+=:\gamma_j^+$, 
and $\delta^+:=\sum_{j=1}^n\delta_j=\sum_{j=1}^n\gamma_j=:\gamma^+$.
The sets $\Omega_j'$ and $\Omega_j''$ are as in Remark \ref{rem:xy}.
\end{rem}

\begin{prop}\label{prop:bassoc}
With the above notation, for 
$\res_{Y_n'}^{\mathfrak{A}_n}(E)=\res_{Y_n'}^{\mathfrak{S}_n}(M'/M'')$, we
have:
\begin{enumerate}
\item[(i)] If $l>2$ then $\{\overline{\delta}_j^+\mid j\in\{1,\ldots,l-1\}\}$ 
is an $F$-basis of $\Soc(\res_{Y_n'}^{\mathfrak{A}_n}(E))$.
\item[(ii)] 
If $l=2$ and $n_l>2$ then $\{\overline{\delta}_0^+,\overline{\delta}_1^+\}$ 
is an $F$-basis of $\Soc(\res_{Y_n'}^{\mathfrak{A}_n}(E))$.
\item[(iii)] 
If $l=2$ and $n_l=2$ then $\res_{Y_n'}^{\mathfrak{A}_n}(E)\cong FY_n'$ 
with socle spanned by $\overline{\delta}_1^+$.
\end{enumerate}
\end{prop}

\begin{proof}
Let $v=\sum_{i=1}^nc_i\delta_i\in M$, for appropriate $c_i\in F$,
such that $\bar{v}\in\Soc(\res_{Y_n'}^{\mathfrak{S}_n}(M/M''))$. 
Since the group $Y_n'$ acts trivially on 
$\Soc(\res_{Y_n'}^{\mathfrak{S}_n}(M/M''))$, for each $j\in\{1,\ldots,l\}$ 
there is some $a_j\in F$ such that
$vy_{n_j}'=v+a_j\delta^+$.

\medskip
Suppose first that $l>2$ so that 
$\Omega\setminus (\Omega_j\cup\Omega_l)\neq \emptyset$, for $j=1,\ldots,l$. 
That is, $\delta_iy_{n_j}'=\delta_i$, for 
$i\in\Omega\setminus (\Omega_j\cup\Omega_l)$. In particular, 
for $j=1,\ldots,l$ and $i\in \Omega\setminus (\Omega_j\cup\Omega_l)$, we get
$c_i=c_i+a_j$. Thus $a_j=0$, for $j=1,\ldots,l$.
If $j\leq l-1$ then $vy_{n_j}'=v$ holds if and only if
$c_{n_1+\cdots+n_{j-1}+1}=\ldots =c_{n_1+\cdots +n_j}$ and 
$c_{n_1+\cdots +n_{l-1}+1}=\cdots =c_{n_1+\cdots +n_l}$. 
Furthermore, $vy_{n_j}'=v$ for $j=1,\ldots,l-1$ implies 
also $vy_{n_l}'=v$. Consequently, we deduce:
$$\bigcap_{j=1}^l\{v\in M\mid vy_{n_j}'=v\}
=\bigcap_{j=1}^{l-1}\{v\in M\mid vy_{n_j}'=v\}
=\langle \delta_1^+,\ldots,\delta_l^+\rangle.$$
Moreover, $\delta_1^+,\ldots,\delta_l^+$ are linearly independent, 
and thus form a basis for 
$\Soc(\res_{Y_n'}^{\mathfrak{S}_n}(M))$. This in turn implies that 
$\{\overline{\delta}_j^+\mid j\in\{1,\ldots,l-1\}\}\subset M'/M''$
is an $F$-basis for
\begin{align*}
\{\bar{v}\in M/M''\mid v\in\langle \delta_1^+,\ldots,\delta_l^+\rangle_F\}&=\{\bar{v}\in M/M''\mid \bar{v}y_{n_j}'=\bar{v}\text{ for } j=1,\ldots,l\}\\
&=\Soc(\res_{Y_n'}^{\mathfrak{S}_n}(M/M'')).
\end{align*}
Hence 
$\Soc(\res_{Y_n'}^{\mathfrak{S}_n}(M'/M''))
=\Soc(\res_{Y_n'}^{\mathfrak{S}_n}(M/M''))
=\langle\overline{\delta}_j^+\mid j=1,\ldots,l-1\rangle_F$,
implying (i). 

\medskip 
Next suppose that $l=2$. If also $j=2$ then, 
since $\Omega\setminus \Omega_2=\Omega_1\neq\emptyset$, 
we get $a_2=0$. 
If $j=1$ then, since $\Omega\setminus (\Omega_1\cup\Omega_2)=\emptyset$, 
we deduce 
$vy_{n_1}'=\sum_{i=1}^nc_{i^{(y_{n_1}'^{-1})}}\delta_i
=\sum_{i=1}^n(c_i+a_1)\delta_i$, 
so that
\begin{equation*}\label{equ:c}
c_i=
\begin{cases}
c_1,& \text{ if } i\in\{1,\ldots,n_1\} \text{ is odd},\\
c_1+a_1,& \text{ if } i\in\{1,\ldots,n_1\} \text{ is even},\\
c_{n_1+1},& \text{ if } i\in\{n_1+1,\ldots,n_1+n_2\} \text{ is odd},\\
c_{n_1+1}+a_1,& \text{ if } i\in\{n_1+1,\ldots,n_1+n_2\} \text{ is even}.
\end{cases}
\end{equation*}
Hence $\{\delta_0^+,\delta_1^+,\delta_2^+\}$ is an
$F$-basis of $\{v\in M\mid vy_{n_1}'=v+a_1\delta^+\}$.
Thus the condition $0\neq\bar{v}\in\Soc(\res_{Y_n'}^{\mathfrak{S}_n}(M/M''))$ 
is equivalent to $vy_{n_2}'=v$ where 
$v \in \langle\delta_0^+,\delta_1^+,\delta_2^+\rangle_F$.
But $vy_{n_2}'=v$ holds if and only if $c_{n_1+1}=c_{n_1+i}$, 
for $i=1,\ldots,n_2-1$ odd, and $c_{n_1+2}=c_{n_1+i}$ for
$i=2,\ldots,n_2$ even. 
That is, $\{\delta_1,\ldots,\delta_{n_1},\delta_{2'}^+,\delta_{2''}^+\}$ is an
$F$-basis of $\{v\in M\mid vy_{n_2}'=v\}$. 
Since
$\langle\delta_0^+,\delta_1^+,\delta_2^+\rangle_F\leq 
\langle\delta_1,\ldots,\delta_{n_1},\delta_{2'}^+,\delta_{2''}^+\rangle_F$,
we get
$\Soc(\res_{Y_n'}^{\mathfrak{S}_n}(M/M''))
=\langle\overline{\delta}_0^+,\overline{\delta}_1^+,
 \overline{\delta}_2^+\rangle_F$. 
Furthermore, $\overline{\delta}_0^+$ and $\overline{\delta}_1^+$ 
are linearly independent over $F$, whereas
$\overline{\delta}_2^+=\overline{\delta}_1^+$. 

\medskip
Therefore, if $n_l>2$ then
$\{\overline{\delta}_0^+,\overline{\delta}_1^+\}\subset M'/M''$ is in fact
a basis of 
$\Soc(\res_{Y_n'}^{\mathfrak{S}_n}(M/M''))
=\Soc(\res_{Y_n'}^{\mathfrak{S}_n}(M'/M''))$,
proving (ii).
If finally $n_l=2$ then we have $\delta_0^+\notin M'$. 
Hence in this case we deduce that 
$\Soc(\res_{Y_n'}^{\mathfrak{S}_n}(M'/M''))$ has $F$-basis
$\{\overline{\delta}_1^+\}$. In particular, 
$\res_{Y_n'}^{\mathfrak{A}_n}(E)$ is indecomposable so that, 
comparing dimensions, we get
$\res_{Y_n'}^{\mathfrak{A}_n}(E)\cong FY_n'$, proving (iii).
\end{proof}

\begin{prop}\label{prop:ker}
Let $l\geq 2$ and $j\in\{1,\ldots,l\}$, and 
let $K_j:=\ker_D(y_{n_j}'^+)$ and $I_j:=\im_D(y_{n_j}'^+)$.
\begin{enumerate}
\item[(i)] If $j<l$ then $\dim(K_j)=\dim(D)-1$, and 
$I_j=\langle\overline{\delta}_j^+\rangle_F$.
\item[(ii)] If $j=l$ and $n_l>2$ then $\dim(K_j)=\dim(D)-2$, and 
$I_j=\langle\overline{\delta}_{l'}^+,\overline{\delta}_{l''}^+\rangle_F$.
\item[(iii)] If $j=l$ and $n_l=2$ then $y_{n_j}'^+=1$ and thus $K_j=0$.
\end{enumerate}
\end{prop}

\begin{proof}
For $j=1,\ldots,l$ we define the following elements in $M'$:
$$v_j:=\delta_{n_1+\cdots +n_j}+\delta_{n_1+\cdots +n_j+1},\text{ if }j<l,$$
as well as $v'_l:=\delta_n+\delta_1$ and $v''_l:=\delta_{n-1}+\delta_1$. 
Then $v_jy_{n_j}'^+=\delta_j^+,\text{ if } j\leq l-1$, and 
$$v'_ly_{n_l}'^+=\begin{cases}
v'_l,& \text{ if } n_l=2,\\
\delta_{l''}^+,&\text{ if } n_l>2,
\end{cases}
\quad\text{ and }\quad
v''_ly_{n_l}'^+=\begin{cases}
v''_l,& \text{ if } n_l=2,\\
\delta_{l'}^+,&\text{ if } n_l>2. \\
\end{cases}$$

\medskip
We first suppose that $j<l$.
Then $Y_{n_j}'$ acts regularly on 
$U_j:=\langle \delta_{n_1+\cdots +n_{j-1}+1},\ldots,
              \delta_{n_1+\cdots +n_j}\rangle_F$, 
so that
$U_jy_{n_j}'^+\cong\Rad^{n_j-1}(FY_{n_j}')\cong F$ and 
$\dim(\ker_{U_j}(y_{n_j}'^+))=n_j-1$. Furthermore, 
the vector space 
$\langle\delta_{n_1+\cdots +n_{l-1}+1},\ldots,
        \delta_{n_1+\cdots +n_l}\rangle_F$ 
is an indecomposable $FY_{n_j}'$-module of dimension $n_l$. 
In particular, 
$\langle\delta_{n_1+\cdots +n_{l-1}+1},\ldots,
        \delta_{n_1+\cdots +n_l}\rangle_F$ is annihilated
by $\Rad^{n_j-1}(FY_{n_j}')=F\langle y_{n_j}'^+\rangle=\Soc(FY_{n_j}')$. 
Furthermore, $Y_{n_j}'$ acts trivially on the subspace
$$\langle \{\delta_1,\ldots,\delta_n\}\setminus
  \{\delta_{n_1+\cdots +n_{j-1}+1},\ldots,\delta_{n_1+\cdots +n_j},
  \delta_{n_1+\cdots + n_{l-1}+1},\ldots,\delta_{n_1+\cdots +n_l}\}\rangle_F$$
of $M$ which is thus annihilated by $y_{n_j}'^+$. 
To summarize, we have shown that
$\dim(\ker_M(y_{n_j}'^+))=n-1=\dim(M')$.
Hence $\dim(\ker_{M'}(y_{n_j}'^+))\in\{\dim(M'),\dim(M')-1\}=\{n-1,n-2\}$ and 
$\dim(\ker_D(y_{n_j}'^+))=\dim(K_j)\in\{\dim(D), \dim(D)-1\}=\{n-2,n-3\}.$
But we already know that 
$0\neq \overline{\delta}_j^+\in\im_D(y_{n_j}'^+)=I_j$ so that 
$I_j=\langle\overline{\delta}_j^+\rangle_F$ and $\dim(K_j)=\dim(D)-1$, 
proving (i).  

\medskip
Now let $j=l$ and $n_l>2$. Then $Y_{n_l}'$ acts trivially on 
$\langle\delta_1,\ldots,\delta_{n_1+\ldots +n_{l-1}}\rangle_F$.
Furthermore, 
$\langle\delta_{n_1+\cdots +n_{l-1}+1},\ldots,
 \delta_{n_1+\ldots +n_l}\rangle_F$ 
is an $FY_{n_l}'$-module isomorphic to $FY_{n_l}'\oplus FY_{n_l}'$. 
Consequently, $\langle\delta_1,\ldots,\delta_{n_1+\ldots +n_{l-1}}\rangle_F$ 
is annihilated by $y_{n_l}'^+$, and 
$$\langle\delta_{n_1+\cdots +n_{l-1}+1},\ldots,
         \delta_{n_1+\ldots +n_l}\rangle_F\cdot y_{n_l}'^+
  \cong\Rad^{n_l-1}(FY_{n_l}')\oplus
       \Rad^{n_l-1}(FY_{n_l}')\cong F\oplus F.$$
Thus $\dim(\ker_M(y_{n_l}'^+))=n-2=\dim(M')-1$ so that 
$\dim(\ker_{M'}(y_{n_l}'^+))\geq \dim(M')-2=n-3$ and 
$\dim(K_l)=\dim(\ker_D(y_{n_l}'^+))\geq \dim(D)-2=n-4$. 
As we have already shown above, 
$\langle\overline{\delta}_{l'}^+,\overline{\delta}_{l''}^+\rangle_F
 \leq \im_D(y_{n_l}'^+)=I_l$. 
Since $l\geq 2$, we conclude that
$\overline{\delta}_{l'}^+$ and $\overline{\delta}_{l''}^+$ 
are linearly independent over $F$, and hence
$I_l=\langle \overline{\delta}_{l'}^+,\overline{\delta}_{l''}^+\rangle_F$ 
and $\dim(K_l)=n-4=\dim(D)-2$. This proves (ii),
and assertion (iii) is obviously true, since $y_{n_l}'=1$ for $n_l=2$.
\end{proof}

\begin{rem}\label{rem:formerly51}\normalfont
Let $l\geq 2$. We consider the subgroup 
$X_n=\prod_{j=1}^l X_{n_j}\leq\Phi(Q_n)\leq Q_n\leq P_n$.
Letting $Q\leq P_n$ be any subgroup such that $X_n\leq Q$,
we show when $X_n$ can be used to 
prove the indecomposability of $\res_Q^{\mathfrak{S}_n}(D)$.
Note that this is slightly more general than needed to prove
Theorem \ref{thm:main}, but will be useful in the proof of
Theorem \ref{thm:correction}.
Therefore, we fix an indecomposable direct sum decomposition of 
$\res_Q^{\mathfrak{S}_n}(D)$.

\medskip
(a) Suppose that $j\in\{1,\ldots,l\}$ such that $n_j>2$,
that is, we exclude just the case $j=l$ and $n_l=2$. Let 
$u:=\gamma_{n_1+\cdots+n_j}+\gamma_{n_1+\cdots +n_j+1}\in M'$ and 
$v:=\gamma_{n_1+\cdots +n_{j-1}+1}+\gamma_{n_1+\cdots +n_{j-1}+n_j/2+1}\in M'$
where the indices should be read modulo $n$. Then
\begin{equation}\label{equ:xuv}
ux_{n_j}^+=\gamma_{j''}^+=\delta_{j''}^+\notin M''
\quad\text{ and }\quad vx_{n_j}^+=\gamma_j^+=\delta_j^+\notin M''.
\end{equation}
Furthermore, 
$\langle\gamma_{n_1+\cdots +n_{j-1}+1},\ldots,
 \gamma_{n_1+\dots +n_j}\rangle_F$ is an 
$FX_{n_j}$-module isomorphic to $FX_{n_j}\oplus FX_{n_j}$, 
and $X_{n_j}$ acts trivially on 
$\langle \gamma_1,\ldots,\gamma_{n_1+\dots +n_{j-1}},
 \gamma_{n_1+\dots +n_j+1},\ldots,\gamma_n\rangle_F$. 
Hence 
$$\im_M(x_{n_j}^+)=\Soc(\langle\gamma_{n_1+\cdots +n_{j-1}+1},\ldots,
  \gamma_{n_1+\dots +n_j}\rangle_F)\cong F\oplus F,$$ 
and $\dim(\ker_M(x_{n_j}^+))=n-2$. 
This implies $\dim(\ker_{M'}(x_{n_j}^+))\geq \dim(M')-2$ and thus
$\dim(\ker_{D}(x_{n_j}^+))\geq \dim(D)-2$. Since, by (\ref{equ:xuv}),
$\overline{\gamma}_{j'}^+,\overline{\gamma}_{j''}^+\in \im_D(x_{n_j}^+)$,
we get $\dim(\ker_D(x_{n_j}^+))=\dim(D)-2=n-4$ and
$\im_D(x_{n_j}^+)=\langle \overline{\gamma}_{j'}^+,
\overline{\gamma}_{j''}^+\rangle_F$.
Hence we have $(FX_{n_j}\oplus FX_{n_j})|\res_{X_{n_j}}^{\mathfrak{S}_n}(D)$,
and either precisely one or precisely two indecomposable 
summands in the fixed direct sum decomposition 
of $\res_Q^{\mathfrak{S}_n}(D)$ are not annihilated by $x_{n_j}^+$.
 
\medskip
(b) Suppose that there are two of these, $U'$ and $U''$, say. 
Recall further that we are still assuming $n_j>2$.
Then 
there are some $a',a'',b',b''\in F$ such that
$0\neq u':=a'\overline{\gamma}_{j'}^++a''\overline{\gamma}_{j''}^+\in U'$
and
$0\neq u'':=b'\overline{\gamma}_{j'}^++b''\overline{\gamma}_{j''}^+\in U''$.
In particular, $u'$ and $u''$ are linearly independent. 
Next consider $Q^{p_j q_j}$, where 
$$p_j:P_n\longrightarrow P_{n_j}\quad \text{ and } \quad 
  q_j:P_{n_j}\longrightarrow P_{n_j}/B_{n_j}\cong\langle w_{n_j,j}\rangle$$ 
are the natural epimorphisms. We have either
$Q^{p_j q_j}=\langle w_{n_j,j}\rangle$ or $Q^{p_j q_j}=\{1\}$. 
Assume that $Q^{p_j q_j}=\langle w_{n_j,j}\rangle$. Then there
exists some $g\in Q$ such that $\gamma_{j'}^+g=\gamma_{j''}^+$ and 
$\gamma_{j''}^+g=\gamma_{j'}^+$. If $a'\neq a''$ then
$u'+u'g=(a'+a'')\overline{\gamma}_j^+\in U'$ and thus also 
$\overline{\gamma}_j^+\in U'$. If $a'=a''$ then
$0\neq a'\overline{\gamma}_j^+\in U'$, 
and again $\overline{\gamma}_j^+\in U'$. But, analogously, we deduce
$\overline{\gamma}_j^+\in U''$, yielding the contradiction 
$0\neq \overline{\gamma}_j^+\in U'\cap U''$. 
Consequently, this forces $Q^{p_j q_j}=\{1\}$, that is $Q\leq\ker(p_j q_j)$.

\medskip
(c) Now suppose that $n_l>2$ and that, for each $j=1,\ldots,l$, 
there is exactly one 
indecomposable direct summand $U_j$ in the fixed direct sum decomposition
of $\res_Q^{\mathfrak{S}_n}(D)$
which is not annihilated by $x_{n_j}^+$. We thus deduce
$\overline{\gamma}_j^+\in \langle\overline{\gamma}_{j'}^+,
 \overline{\gamma}_{j''}^+\rangle_F=\im_D(x_{n_j}^+)\leq U_j$,
for $j=1,\ldots,l$. Since 
$\sum_{j=1}^l\overline{\gamma}_j^+=\overline{\gamma}^+=0$, 
this implies $U_1=\ldots =U_l$. 

\medskip
Next we show that 
$\Soc(\res_{X_n}^{\mathfrak{S}_n}(D))\leq 
 \res_{X_n}^Q(U_1)+\cdots +\res_{X_n}^Q(U_l)=\res_{X_n}^Q(U_1)$, 
which then implies that $\res_{Q}^{\mathfrak{S}_n}(D)=U_1$ is indecomposable.
For this, let $v\in M'$. Then 
$\overline{v}\in\Soc(\res_{X_n}^{\mathfrak{S}_n}(D))$ 
if and only if $\overline{v}x_{n_j}=\overline{v}$, for all $j=1,\ldots,l$, 
that is, if and only if for each $j\in\{1,\ldots,l\}$ 
there is some $a_j\in F$ such that $vx_{n_j}=v+a_j\gamma^+$.
Let $j\in\{1,\ldots,l\}$.
Since $X_{n_j}$ acts trivially on $\Omega\setminus \Omega_j\neq \emptyset$, 
this condition is actually equivalent to $vx_{n_j}=v$. 
That is, $\overline{v}\in D$ is fixed by $X_{n_j}$ if and only $v\in M'$ is. 
Since
$\langle\gamma_{n_1+\cdots +n_{j-1}+1},\ldots,\gamma_{n_1+\dots +n_j}
 \rangle_F$ 
is isomorphic to $FX_{n_j}\oplus FX_{n_j}$ as an $FX_{n_j}$-module, we get
$$\Soc(\res_{X_n}^{\mathfrak{S}_n}(M))=\{v\in M\mid vx_{n_j}=v 
  \text{ for all } j=1,\ldots,l\}=\bigoplus_{j=1}^l
  \langle\gamma_{j'}^+,\gamma_{j''}^+\rangle_F
  \leq \res_{X_n}^{\mathfrak{S}_n}(M').$$
Consequently, 
$\Soc(\res_{X_n}^{\mathfrak{S}_n}(M))=
 \Soc(\res_{X_n}^{\mathfrak{S}_n}(M'))=
 \bigoplus_{j=1}^l\langle\gamma_{j'}^+,\gamma_{j''}^+\rangle_F$, 
and thus indeed
$$\Soc(\res_{X_n}^{\mathfrak{S}_n}(D))
=\langle \overline{\gamma}_{1'}^+,\overline{\gamma}_{1''}^+,\ldots,
\overline{\gamma}_{l'}^+,\overline{\gamma}_{l''}^+\rangle_F\leq 
\res_{X_n}^Q(U_1)+\cdots +\res_{X_n}^Q(U_l).$$
\end{rem}

\section{The case $l\geq 2$ and $n_l>2$}\label{sec:nlgt2}

The behaviour of the natural $F\mathfrak{A}_n$-module
$E$ upon restriction to the abelian subgroup $Y_n'$ of $Q_n$ depends
on whether $n_l>2$ or $n_l=2$. Therefore, we will now distinguish
between these two cases, and start off with the case $n_l>2$.

\begin{prop}
Let $l\geq 2$ and $n_l>2$. Then $E$ has vertex $Q_n$ and source
$\res_{Q_n}^{\mathfrak{A}_n}(E)$.
\end{prop}

\begin{proof}
We follow the strategy given in Remark \ref{rem:strategy},
and assume that $R<Q_n$ is a maximal subgroup such that $E$ is relatively 
$R$-projective, hence
$X_n=\prod_{j=1}^l X_{n_j}\leq\Phi(Q_n)\leq R$ and $Q_{n-4}\leq R$.  

\medskip
First of all, we fix an indecomposable direct sum decomposition of
$\res_{Q_n}^{\mathfrak{A}_n}(E)$.
Then, for each $j=1,\ldots,l$, we have 
$Q_n^{p_j q_j}=\langle w_{n_j,j}\rangle \neq 1$;
note that here we need $n_l\geq 4$.
Therefore, in consequence of Remark \ref{rem:formerly51}, part (b),
we obtain that, for each $j=1,\ldots,l$, there is precisely 
one indecomposable direct summand $U_j$ in the decomposition 
which is not annihilated by $x_{n_j}^+$. 
Hence, Remark \ref{rem:formerly51}, part (c), shows that 
$\res_{Q_n}^{\mathfrak{A}_n}(E)$ is, in fact, indecomposable.

\medskip
Now we fix an indecomposable direct sum decomposition of
$\res_R^{\mathfrak{A}_n}(E)$. Hence, by Remark \ref{rem:formerly51}, 
part (a), there is either one indecomposable direct summand $U$ or
there are exactly two indecomposable direct summands $U'$ and $U''$
in the decomposition which are not annihilated by $x_{n_1}^+$. 
In the first case we get
$(FX_{n_1}\oplus FX_{n_1})|\res_{X_{n_1}}^{R}(U)$, hence 
$\dim(U)\geq 2\dim(FX_{n_1})=n_1\geq (n+2)/2>(n-2)/2=\dim(E)/2$,
a contradiction.
In the second case Remark \ref{rem:formerly51}, part (b), 
implies that $R\leq\ker(p_1 q_1)$.
But, since $l\geq 2$ and $n_l\geq 4$, we have 
$w_{n_1,1}\in Q_{n_1}\leq Q_{n-4}\leq R$ and $w_{n_1,1}\notin\ker(p_1 q_1)$, 
a contradiction.
\end{proof} 

\medskip\normalfont
It remains to deal with the case $n_l=2$. It turns out that $E$
behaves differently upon restriction to $Y_n'$, depending on the value
of $l\geq 2$. In the following section we begin with the generic case 
$l\geq 4$, while the exceptional cases $l\leq 3$ will be settled in 
subsequent sections.


\section{The case $n_l=2$ and $l\geq 4$}\label{sec:nl2lgeq4}
\normalfont

\begin{lemma}\label{lemma:l4nl2indec}
Let $n_l=2$ and $l\geq 4$. 
Then $\res_{Y_n'}^{\mathfrak{A}_n}(E)$ is indecomposable.
\end{lemma}

\begin{proof}
We fix an indecomposable direct sum decomposition of 
$\res_{Y_n'}^{\mathfrak{A}_n}(E)$. For $j\in\{1,\ldots,l-1\}$, 
by Proposition \ref{prop:ker}, there is precisely one 
indecomposable direct summand $V_j$ in this decomposition
which is not annihilated by $y_{n_j}'^+$.
We denote the preimage of $V_j$ under the natural epimorphism
$M'\longrightarrow M'/M''$ by $\hat{V}_j$.
From $\overline{\gamma}_j^+\in V_j$, for $j=1,\ldots,l-1$, 
we by Proposition \ref{prop:bassoc} conclude that
$\Soc(\res_{Y_n'}^{\mathfrak{A}_n}(E))\leq V_1+\cdots +V_{l-1}$,
so that $\res_{Y_n'}^{\mathfrak{A}_n}(E)=V_1+\cdots +V_{l-1}$.
Hence it suffices to show that $V_1=\ldots =V_{l-1}$. 
Suppose, for a contradiction,
that this is not the case. After appropriate re-numbering, we may suppose that
$V_1=\ldots =V_k$ and $V_{k+1}\neq V_1,\ldots, V_{l-1}\neq V_1$.
Furthermore, since $l\geq 4$, we may also suppose that $k<l-2$.

\medskip
For $i=1,\ldots, l$, we have the $FP_n$-epimorphism
$$\pi_i:M\longrightarrow M_{n_i}:=\langle
\gamma_{n_1+\cdots +n_{i-1}+1},\ldots,\gamma_{n_1+\cdots +n_i}\rangle_F,$$
and we consider 
$\pi:=\bigoplus_{j=1}^k\pi_j\oplus \pi_l: 
M\longrightarrow \bigoplus_{j=1}^kM_{n_j}\oplus M_{n_l} $.
For $j=1,\ldots,l$ and $v\in M_{n_j}$, set 
$h_j(v):=\min\{e\in\mathbb{N}_0\mid v\in \ker_{M_{n_j}}((y_{n_j}+1)^e)\}$;
note that
$h_j(v)=\min\{e\in\mathbb{N}_0\mid v\in \ker_{M_{n_j}}((y_{n_j}'+1)^e)\}$ 
whenever $j\leq l-1$.  

\medskip
Now assume that there is some $0\neq\tilde{v}\in\hat{V}_1\cap\ker(\pi)$. 
If $h_j(\tilde{v}^{\pi_j})\leq 1$ for all $j\in\{k+1,\ldots,l-1\}$,
then $\tilde{v}^{\pi_j}=a_j\gamma_j^+$, for some $a_j\in F$, and hence 
$0\neq\tilde{v}\in \hat{V}_1\cap (\hat{V}_{k+1}+\cdots +\hat{V}_{l-1})
=\langle\gamma^+\rangle_F$.
If there is some $j\in\{k+1,\ldots,l-1\}$ 
such that $h_j(\tilde{v}^{\pi_j})>1$, then let
$v:=\tilde{v}(y_{n_j}'+1)^{h_j(\tilde{v}^{\pi_j})-1}
 \in (\hat{V}_1\cap\ker(\pi))\setminus\{0\}$,
hence $v^{\pi_r}=0$ for $r\in\{k+1,\ldots,l-1\}\setminus\{j\}$, 
and $v^{\pi_j}=a\gamma_j^+$, for some $0\neq a\in F$,
thus $0\neq v\in \hat{V}_1\cap \hat{V}_j=\langle\gamma^+\rangle_F$. 
In both cases, $\gamma^+\notin\ker(\pi)$ yields a contradiction. 
Consequently, $\pi_{|\hat{V}_1}$ is injective.

\medskip
Next we consider the $FP_n$-epimorphism 
$\pi':=\bigoplus_{j=1}^k\pi_j:M\longrightarrow\bigoplus_{j=1}^kM_{n_j}$.
Assume that there
is some $0\neq\tilde{v}\in \hat{V}_1\cap\ker(\pi')$. 
We write $\tilde{v}=\sum_{i=1}^na_i\delta_i$, for appropriate
$a_1,\ldots, a_n\in F$. By the injectivity of $\pi_{|\hat{V}_1}$ we have
$\tilde{v}\notin\ker(\pi)$, that is, $a_{n-1}\neq 0$ or $a_n\neq 0$. We set 
$v_0:=\sum_{j=k+1}^l\gamma_j^+=\sum_{j=k+1}^l\delta_j^+$, 
then $v_0\in\ker(\pi')$. 
Moreover, $v_0=\gamma^++\sum_{j=1}^k\gamma_j^+$, and thus
also $v_0\in\hat{V}_1$. 
If $a_{n-1}=a_n\neq 0$, then we may suppose that $a_{n-1}=1=a_n$, and thus
$$ \tilde{v}+v_0=\sum_{j=n_1+\cdots+n_k+1}^{n_1+\cdots +n_{l-1}}(a_j+1)\delta_j
   \in\hat{V}_1\cap\ker(\pi)=\{0\},$$
hence $\tilde{v}=\sum_{j=k+1}^l\gamma_j^+=v_0$. 
Now assume that $a_{n-1}\neq a_n$. 
Let $j\in\{k+1,\ldots,l-1\}$. Then
\begin{align*}
0\neq v&:=\tilde{v}(y_{n_j}'+1)
=(\sum_{i=n_1+\cdots +n_{j-1}+2}^{n_1+\cdots +n_j}(a_i+a_{i-1})\delta_i)\\
&
+(a_{n_1+\cdots+n_{j-1}+1}+a_{n_1+\cdots +n_j})\delta_{n_1+\cdots +n_{j-1}+1}
+(a_{n-1}+a_n)(\delta_{n-1}+\delta_n)\in\hat{V}_1\cap\ker(\pi').
\end{align*}
Hence we have $v=bv_0$, for some $0\neq b\in F$, 
which in turn means that $j=k+1=l-1$. 
But, since $k<l-2$, this is a contradiction.
Hence we actually have $a_{n-1}=a_n\neq 0$,
and thus $\hat{V}_1\cap\ker(\pi')=\langle v_0\rangle_F$. 

\medskip
We now show that $\pi'_{|\hat{V}_1}$ is surjective:
For $j\in\{k+1,\ldots,l-1\}$ and $i\in\{1,\ldots,k\}$
we have $\hat{V}_j^{\pi_i}\leq\Rad(\res_{Y_{n_i}'}^{P_n}(M_{n_i}))$,
since otherwise taking 
$w\in\hat{V}_j$ such that
$w^{\pi_i}\in M_{n_i}\setminus\Rad(\res_{Y_{n_i}'}^{P_n}(M_{n_i}))$ 
yields $w(y_{n_i}'+1)^{n_i-1}=a\gamma_i^+\in\hat{V}_j$, for
some $0\neq a\in F$,
a contradiction.
Hence we have 
$$ (\hat{V}_{k+1}+\cdots+\hat{V}_{l-1})^{\pi'}\leq
\bigoplus_{i=1}^k\Rad(\res_{Y_{n_i}'}^{P_n}(M_{n_i}))
=\Rad(\res_{Y_n'}^{P_n}(\bigoplus_{i=1}^k M_{n_i})) .$$
From $\hat{V}_1+\cdots+\hat{V}_{l-1}=\res_{Y_n'}^{\mathfrak{S}_n}(M')$ we get
$(\hat{V}_1+\hat{V}_{k+1}+\cdots+\hat{V}_{l-1})^{\pi'}
=\res_{Y_n'}^{P_n}(\bigoplus_{i=1}^k M_{n_i})$,
hence $\hat{V}_1^{\pi'}=\res_{Y_n'}^{P_n}(\bigoplus_{i=1}^k M_{n_i})$.

\medskip
Thus there is some $w\in\hat{V}_1$ such that 
$w^{\pi'}=\gamma_1=\delta_1$. We write
$w=\delta_1+\sum_{i=n_1+\cdots +n_k+1}^n a_i\delta_i$, 
for some $a_{n_1+\cdots +n_k+1},\ldots,a_n\in F$. 
Furthermore, $\sum_{i=n_1+\cdots n_k+1}^na_i=1$, 
since $w\in\hat{V}_1\leq M'$.
Assume there is $j\in\{k+1,\ldots,l-1\}$
such that $h_j(w^{\pi_j})>1$. Then let 
$$ \tilde{w}:=w(y_{n_j}'+1)\in(\hat{V}_1\cap\ker(\pi'))\setminus\{0\}. $$
Hence we have $\tilde{w}^{\pi_j}\neq 0$, 
and $\tilde{w}^{\pi_r}=0$ for $j\neq r\in\{k+1,\ldots,l-1\}$,
while $\tilde{w}^{\pi_l}=(a_{n-1}+a_n)(\delta_{n-1}+\delta_n)$.
But $\hat{V}_1\cap\ker(\pi')=\langle v_0\rangle_F$ 
and our hypothesis $k<l-2$ yield a contradiction,
thus we have $h_j(w^{\pi_j})\leq 1$ for all $j\in\{k+1,\ldots,l-1\}$.
In particular, $\sum_{i=n_1+\cdots +n_{j-1}+1}^{n_1+\cdots +n_j}a_i=0$, for
$j\in\{k+1,\ldots,l-1\}$.
This forces $a_{n-1}+a_n=1$, that is $h_l(w^{\pi_l})=2$.
Thus, for any $j\in\{k+1,\ldots,l-1\}$, we get
$w(y_{n_j}'+1)=(a_{n-1}+a_n)(\delta_{n-1}+\delta_n)
\in(\hat{V}_1\cap\ker(\pi'))\setminus\{0\}$, contradicting
$\hat{V}_1\cap\ker(\pi')=\langle v_0\rangle_F$.
\end{proof}

\begin{prop}\label{prop:l4nl2}
Let $n_l=2$ and $l\geq 4$.
Then $E$ has vertex $Q_n$ and source $\res_{Q_n}^{\mathfrak{A}_n}(E)$.
\end{prop}

\begin{proof}
We follow the strategy given in Remark \ref{rem:strategy}.
By Lemma \ref{lemma:l4nl2indec},
we already know that $\res_{Q_n}^{\mathfrak{A}_n}(E)$ is indecomposable.
Assume that $R<Q_n$ is a maximal subgroup such that $E$ is relatively
$R$-projective, hence
$X_n=\prod_{j=1}^l X_{n_j}\leq\Phi(Q_n)\leq R$ and $Q_{n-4}\leq R$.
We fix an indecomposable direct sum decomposition of
$\res_R^{\mathfrak{A}_n}(E)$.
By Remark \ref{rem:formerly51}, part (a),
there is either exactly one summand $U$ or there are exactly two summands
$U'$ and $U''$ in this decomposition which are 
not annihilated by $x_{n_1}^+$.
In the first case
$(FX_{n_1}\oplus FX_{n_1})|\res_{X_{n_1}}^{R}(U)$ implies
$\dim(U)\geq n_1>\dim(E)/2$, a contradiction.
In the second case,
Remark \ref{rem:formerly51}, part (b), implies that
$R\leq\ker(p_1 q_1)$. But since $l\geq 3$ and $n_l=2$ we have
$Q_{n-4}=(P_{n_1}\times\cdots\times P_{n_{l-2}}\times 
 P_{2^{i_{l-1}-1}}\times P_{2^{i_{l-1}-2}}\times\cdots 
 \times P_4\times P_2)\cap Q_n$,
thus $w_{n_1,1}\in Q_{n_1}\leq Q_{n-4}\leq R$, a contradiction.
\end{proof}


\section{The case $n_l=2$ and $l=3$}\label{sec:nl2l3}
\normalfont

\begin{lemma}\label{lemma:l3nl2}
Let $n_l=2$ and $l=3$. Then $\res_{Y_n'}^{\mathfrak{A}_n}(E)=U_1\oplus U_2$
where $U_1$ and $U_2$ are indecomposable of dimension $n_1$ and $n_2$, 
respectively, and both have vertex $Y_n'$.
\end{lemma}

\begin{proof}
We construct the following subspaces $M_1$ and $M_2$ of $M'$: 
$M_1$ has $F$-basis $\mathfrak{B}_1:=\{b_1,\ldots,b_{n_1}\}$ with
$$b_1:=\delta_1+\delta_{n-1}+\sum_{i=0}^{n_2/2-1}\delta_{n_1+1+2i}
\quad\text{ and }\quad
b_j:=\delta_j+\delta_{j-1}+\delta_{n-1}+\delta_{n} ,$$
for $j=2,\ldots,n_1$. 
Obviously, $\mathfrak{B}_1$ is indeed linearly independent. Moreover, 
$M_1+\langle\delta^+\rangle_F$ is an $FY_n'$-module, since
$$ b_jy_{n_1}'=\begin{cases}
b_1+b_2,&\text{for } j=1,\\
b_{j+1},&\text{for } j=2,\ldots,n_1-1,\\
\sum_{i=2}^{n_1}b_i,&\text{for } j=n_1,
\end{cases} \,\,\,
b_jy_{n_2}'=\begin{cases}
b_1+\sum_{i=1}^{n_1/2}b_{2i}+\delta^+,&\text{for } j=1,\\
b_j,&\text{for } j=2,\ldots,n_1.
\end{cases} $$
Similarly, we define the subspace $M_2$ of $M'$ 
with basis $\mathfrak{B}_2:=\{\tilde{b}_1,\ldots,\tilde{b}_{n_2}\}$ where
$$ \tilde{b}_1:=\delta_{n_1+1}+\delta_{n-1}+\sum_{i=0}^{n_1/2-1}\delta_{1+2i}
\quad\text{ and }\quad
\tilde{b}_j:=\delta_{n_1+j}+\delta_{n_1+j-1}+\delta_{n-1}+\delta_{n}, $$
for $j=2,\ldots,n_2$.
Also $M_2+\langle\delta^+\rangle_F$ is an $FY_n'$-module and, by construction, $M_1+M_2+\langle\delta^+\rangle_F$
is an $FY_n'$-submodule of $\res_{Y_n'}^{\mathfrak{S}_n}(M')$. 
We show that $M_1+M_2+\langle\delta^+\rangle_F$
contains an $F$-basis of $M'$,
implying that we actually have
$M_1+M_2+\langle\delta^+\rangle_F=\res_{Y_n'}^{\mathfrak{S}_n}(M')$:
First of all, we have
$\delta_1^+=\sum_{i=1}^{n_1/2}b_{2i}\in M_1$ and
$\delta_2^+=\sum_{i=1}^{n_2/2}\tilde{b}_{2i}\in M_2$.
Thus 
$\delta_{n-1}+\delta_n=\delta_1^++\delta_2^++\delta^+
 \in M_1+M_2+\langle\delta^+\rangle_F$, 
and
$$\delta_i+\delta_{i+1}=\begin{cases}
b_{i+1}+\delta_{n-1}+\delta_n\in M_1+M_2+\langle\delta^+\rangle_F,&
\text{ for } i=1,\ldots,n_1-1,\\
\tilde{b}_{i+1}+\delta_{n-1}+\delta_n\in M_1+M_2+\langle\delta^+\rangle_F,&
\text{ for } i=n_1+1,\ldots,n_1+n_2-1.
\end{cases}$$
Hence we have
$s_1:=\sum_{i=0}^{n_2/2-1}\delta_{n_1+1+2i}
\in M_1+M_2+\langle\delta^+\rangle_F$ 
and $s_2:=\sum_{i=0}^{n_1/2-1}\delta_{1+2i}
\in M_1+M_2+\langle\delta^+\rangle_F$,
thus we get 
$\delta_{1}+\delta_{n-1}=b_1+s_1\in M_1+M_2+\langle\delta^+\rangle_F$ and
$\delta_{n_1+1}+\delta_{n-1}
=\tilde{b}_1+s_2\in M_1+M_2+\langle\delta^+\rangle_F$,
therefore,
$$ \res_{Y_n'}^{\mathfrak{A}_n}(E)
=\res_{Y_n'}^{\mathfrak{S}_n}(M'/M'')
=(M_1+M''/M'')+(M_2+M''/M''). $$
Letting $U_i:=M_i+M''/M''$, we from
$\dim(E)=n-2=n_1+n_2=\dim(M_1)+\dim(M_2)$ get 
$U_i\cong M_i$, for $i\in\{1,2\}$, and hence
$\res_{Y_n'}^{\mathfrak{A}_n}(E)=U_1\oplus U_2$.
For $j=1,\ldots,n_1$ we set $\overline{b}_j:=b_j+M''$, 
and for $j=1,\ldots,n_2$ we set
$\overline{\tilde{b}}_j:=\tilde{b}_j+M''$. Then 
$\overline{\mathfrak{B}}_1:=\{\overline{b}_1,\ldots,\overline{b}_{n_1}\}$
and 
$\overline{\mathfrak{B}}_2
 :=\{\overline{\tilde{b}}_{1},\ldots,\overline{\tilde{b}}_{n_2}\}$ 
are bases for $U_1$ and $U_2$, respectively.
Furthermore, both $U_1$ and $U_2$ are indecomposable. Namely,
$b_1y_{n_1}'^+=\delta_1^+=\sum_{i=1}^{n_1/2}b_{2i}\neq 0$ 
so that $M_1$ is not annihilated
by $y_{n_1}'^+$. In particular, $FY_{n_1}'\mid\res_{Y_{n_1}'}^{Y_n'}(M_1)$. 
Comparing dimensions, we deduce
$FY_{n_1}'\cong\res_{Y_{n_1}'}^{Y_n'}(M_1)$. 
In particular, $\res_{Y_{n_1}'}^{Y_n'}(M_1)$ and thus also $M_1\cong U_1$ 
is uniserial, hence indecomposable. 
The indecomposability of $U_2$ is proved analogously. 

\medskip
Next we show that $Y_n'$ is a vertex of both $U_1$ and $U_2$. 
For this, notice that $Y_n'$ possesses precisely three
maximal subgroups, these are
$Z_1:=\langle y_{n_1}',(y_{n_2}')^2\rangle$,
$Z_2:=\langle y_{n_1}'y_{n_2}',(y_{n_2}')^2\rangle$ and
$Z_3:=\langle(y_{n_1}')^2,y_{n_2}'\rangle$.
We show that neither $U_1$ nor $U_2$ can be
relatively $Z_i$-projective, for $i=1,2,3$.
By Green's Indecomposability Theorem,
it suffices to verify that $U_1$ and $U_2$ 
restrict indecomposably to each of these groups. We investigate $U_1$ first. 
By definition, $(y_{n_2}')^2$ acts trivially on $U_1$. 
That is, we may view $U_1$ as a module for the factor group
$\overline{Y}_n':=Y_n'/\langle (y_{n_2}')^2\rangle$ and show that
it restricts indecomposably to each of the maximal subgroups of 
$\overline{Y}_n'$. The latter are
in natural bijection with the maximal subgroups of $Y_n'$, thus are
$\overline{Z}_1=\langle \overline{y}_{n_1}'\rangle$,
$\overline{Z}_2=\langle \overline{y}_{n_1}'\overline{y}_{n_2}'\rangle$ and
$\overline{Z}_3=\langle (\overline{y}_{n_1}')^2,\overline{y}_{n_2}'\rangle$,
where $^{-}:Y_n'\longrightarrow \overline{Y}_n'$ 
denotes the natural epimorphism. As we have already mentioned, 
$\res_{Y_{n_1}'}^{Y_n'}(U_1)\cong FY_{n_1}'$. Hence also 
$\res_{\overline{Z}_1}^{\overline{Y}_n'}(U_1)\cong F\overline{Z}_1$ 
which is indecomposable.

\medskip
We consider 
$\res_{\overline{Z}_2}^{\overline{Y}_n'}(U_1)$ and
$\res_{\overline{Z}_3}^{\overline{Y}_n'}(U_1)$. 
For convenience, we replace the $F$-basis
$\overline{\mathfrak{B}}_1=\{\overline{b}_1,\ldots,\overline{b}_{n_1}\}$ 
of $U_1$ by 
$\overline{\mathfrak{B}}_1':=\{\overline{b}_1',\ldots,\overline{b}_{n_1}'\}$ 
where $\overline{b}_1':=\overline{b}_1$, and
$\overline{b}_j':=\overline{b}_{j-1}'+\overline{b}_j$, for $j=2,\ldots,n_1$. 
In other words,
$$\overline{b}_j'=\overline{\delta}_j+\sum_{i=0}^{n_2/2-1}
  \overline{\delta}_{n_1+1+2i}+\begin{cases}
\overline{\delta}_{n-1},&\text{ if } j \text{ is odd},\\
\overline{\delta}_{n},&\text{ if } j \text{ is even}.
\end{cases}$$
With this notation, we get
\begin{align*}
\overline{b}_j'\overline{y}_{n_1}'&
=\overline{b}_j'y_{n_1}'=\overline{b}_{j+1}',\text{ for }j=1,\ldots,n_1-1,\\
\overline{b}_{n_1}'\overline{y}_{n_1}'&
=\overline{b}_{n_1}'y_{n_1}'=\overline{b}_1',\\
\overline{b}_j'\overline{y}_{n_2}'&
=\overline{b}_j'y_{n_2}'=\sum_{i\neq j}\overline{b}_i',
\text{ for } j=1,\ldots,n_1.
\end{align*}
We set $\overline{Y}_{n_1}'':=\langle \overline{y}_{n_1}'^2\rangle$,
and notice that $\overline{Y}_{n_1}''\leq \overline{Z}_2$ and
$\overline{Y}_{n_1}''\leq \overline{Z}_3$.
Since 
$\res_{\overline{Y}_{n_1}'}^{\overline{Y}_n'}(U_1)\cong F\overline{Y}_{n_1}'$, 
we also have
$\res_{\overline{Y}_{n_1}''}^{\overline{Y}_n'}(U_1)\cong 
 F\overline{Y}_{n_1}''\oplus F\overline{Y}_{n_1}''$.
More precisely, 
$\res_{\overline{Y}_{n_1}''}^{\overline{Y}_n'}(U_1)=V_1\oplus V_2$ where 
$V_1:=\langle\overline{b}_i'\mid i\text{ odd }\rangle_F$ 
and $V_2:=\langle\overline{b}_i'\mid i\text{ even }\rangle_F$, and the maps
$V_1\longrightarrow F\overline{Y}_{n_1}'',\; 
 \overline{b}_{2i-1}'\longmapsto\overline{y}_{n_1}'^i$
and 
$V_2\longrightarrow F\overline{Y}_{n_1}'',\; 
\overline{b}_{2i}'\longmapsto\overline{y}_{n_1}'^i$,
for $i=1,\ldots,n_1/2$, are isomorphisms of $F\overline{Y}_{n_1}''$-modules.
Consequently, 
$$\mathcal{E}:=\End_{F\overline{Y}_{n_1}''}
  (\res_{\overline{Y}_{n_1}''}^{\overline{Y}_n'}(U_1))
  \cong\Mat(2,F\overline{Y}_{n_1}''),$$
and we will from now on simply identify 
$\mathcal{E}$ and $\Mat(2,F\overline{Y}_{n_1}'')$.
Moreover, $\mathcal{E}$ is a $\overline{Y}_n'$-algebra 
with respect to the conjugation action induced by the
natural embedding
$F\overline{Y}_n'\longrightarrow \mathcal{E}$.
Via this embedding, the elements $\overline{y}_{n_1}'^2$ and 
$\overline{y}_{n_1}'$ correspond to the endomorphisms
$$\begin{pmatrix} \overline{y}_{n_1}'^2&0\\0& \overline{y}_{n_1}'^2 
  \end{pmatrix}\quad \text{ and }\quad
\begin{pmatrix}0&1\\ \overline{y}_{n_1}'^2&0  \end{pmatrix},$$
respectively.
Now $\overline{y}_{n_2}'$ acts on $U_1$ as 
$1+(\overline{y}_{n_1}')^+=1+(\overline{y}_{n_1}'^2)^+
+\overline{y}_{n_1}'\cdot (\overline{y}_{n_1}'^2)^+$.
Letting $s:=(\overline{y}_{n_1}'^2)^+\in\Soc(F\overline{Y}_{n_1}'')$, 
the elements $\overline{y}_{n_2}'$ 
and $\overline{y}_{n_1}'\overline{y}_{n_2}'$
correspond to the endomorphisms
$$\psi:=\begin{pmatrix} 1+s & s\\ s& 1+s  \end{pmatrix}
\quad \text{ and }\quad 
\psi':=\begin{pmatrix} s & 1+s\\ \overline{y}_{n_1}'^2+s& s \end{pmatrix},$$
respectively.

\medskip
We now determine the elements in $\mathcal{E}$ which are fixed by 
$\overline{y}_{n_2}'$.
These are precisely the elements in $\End_{F\overline{Z}_3}(U_1)$. Let
$\varphi\in\mathcal{E}$ such that 
$\varphi=\begin{pmatrix} a& b\\c& d \end{pmatrix}$, for some 
$a,b,c,d\in F\overline{Y}_{n_1}''$. 
Then $\varphi\in\End_{F\overline{Z}_3}(U_1)$ if and only if 
$\varphi\psi=\psi\varphi$, or equivalently, if $(a+d)s=0=(b+c)s$.
That is,
$$\mathcal{F}:=\End_{F\overline{Z}_3}(U_1)
=\left \{\begin{pmatrix} a&b\\c&d \end{pmatrix}
\in \Mat(2,F\overline{Y}_{n_1}'')\mid (a+d)s=0=(b+c)s\right \}.$$
Obviously, $I:=\mathcal{F}\cap\Rad(\Mat(2,F\overline{Y}_{n_1}''))$ 
is a nilpotent ideal in $\mathcal{F}$ such that
\begin{align*} 
\mathcal{F}/I&=\mathcal{F}/\mathcal{F}\cap
\Rad(\Mat(2,F\overline{Y}_{n_1}''))\cong
\mathcal{F}+\Rad(\Mat(2,F\overline{Y}_{n_1}''))/
\Rad(\Mat(2,F\overline{Y}_{n_1}''))\\
&\subseteq\Mat(2,F\overline{Y}_{n_1}'')/
\Rad(\Mat(2,F\overline{Y}_{n_1}'')).
\end{align*}
Since $F\overline{Y}_{n_1}''$ is a local $F$-algebra,
there is an isomorphism 
$$ \Mat(2,F\overline{Y}_{n_1}'')/\Rad(\Mat(2,F\overline{Y}_{n_1}''))
\longrightarrow\Mat(2,F).$$
Moreover, since the annihilator of $s$ in $F\overline{Y}_{n_1}''$
coincides with $\Rad(F\overline{Y}_{n_1}'')$,
the above isomorphism maps the algebra $\mathcal{F}/I$ onto
$$\left\{\begin{pmatrix}\bar{a}& \bar{b}\\\bar{c}&\bar{d}\end{pmatrix}
\in\Mat(2,F)\mid \bar{a}=\bar{d}, \bar{b}=\bar{c}\right\}\cong FC_2.$$
In consequence,
$\mathcal{F}/I$ and thus also $\mathcal{F}=\End_{F\overline{Z}_3}(U_1)$ 
is a local $F$-algebra, and 
$\res_{\overline{Z}_3}^{\overline{Y}_n'}(U_1)$ is indecomposable. 


\medskip
Similarly, $\varphi\in\End_{F\overline{Z}_2}(U_1)$ if and only if
$\varphi\psi'=\psi'\varphi$, or equivalently, if
$(a+d)(s+1)=(a+d)(s+\overline{y}_{n_1}'^2)=0$
and $c(s+1)=b(s+\overline{y}_{n_1}'^2)$.
Since both $s+1,s+\overline{y}_{n_1}'^2\in F\overline{Y}_{n_1}''$ are units, 
we get
$$\mathcal{F}':=\End_{F\overline{Z}_2}(U_1)
=\left \{\begin{pmatrix} a&b\\c&d \end{pmatrix}
\in\Mat(2,F\overline{Y}_{n_1}'')\mid a=d,c=b\overline{y}_{n_1}'^2\right\}.$$
Let 
$0\neq e=a\cdot\begin{pmatrix} 1&0\\0&1 \end{pmatrix} 
        +b\cdot\begin{pmatrix} 0&1\\\overline{y}_{n_1}'^2&0 \end{pmatrix}
\in\mathcal{F}'$,
for some $a,b\in F\overline{Y}_{n_1}''$, be an idempotent. Then 
$$ 
 a\cdot\begin{pmatrix} 1&0\\0&1 \end{pmatrix}
+b\cdot\begin{pmatrix} 0&1\\\overline{y}_{n_1}'^2&0 \end{pmatrix} 
=e=e^2=a^2\cdot\begin{pmatrix} 1&0\\0&1 \end{pmatrix}
      +b^2\cdot\begin{pmatrix} \overline{y}_{n_1}'^2&0\\
                               0&\overline{y}_{n_1}'^2 \end{pmatrix} $$
shows that $b=0$ and $a^2=a\neq 0$ is an idempotent in $F\overline{Y}_{n_1}''$.
Since $F\overline{Y}_{n_1}''$ is a local $F$-algebra we conclude that
$a=1$, thus we have $e=1$, implying that 
$\mathcal{F}'=\End_{F\overline{Z}_2}(U_1)$ is a local $F$-algebra as well,
and $\res_{\overline{Z}_2}^{\overline{Y}_n'}(U_1)$ is indecomposable.

\medskip
Therefore, we have now shown that $U_1$ is an indecomposable 
$F\overline{Y}_n'$-module with vertex $\overline{Y}_n'$.
This in turn implies that $U_1$, as $FY_n'$-module, 
is also indecomposable with vertex $Y_n'$ (cf. \cite{Ku}). Replacing
$y_{n_1}'$ by $y_{n_2}'$, we deduce that also $U_2$ has vertex $Y_n'$. 
\end{proof}

\begin{prop}\label{prop:l3nl2}
Let $n_l=2$ and $l=3$.
Then $E$ has vertex $Q_n$ and source $\res_{Q_n}^{\mathfrak{A}_n}(E)$.
\end{prop}

\begin{proof}
We first show that $\res_{Q_n}^{\mathfrak{A}_n}(E)$ is indecomposable. 
For this, fix an indecomposable direct sum decomposition of 
$\res_{Q_n}^{\mathfrak{A}_n}(E)$. 
For $j\in\{1,2\}$, by Proposition \ref{prop:ker}, there is precisely one
indecomposable direct summand $V_j$ in this decomposition
which is not annihilated by $y_{n_j}'^+$.
We denote the preimage of $V_j$ under the natural epimorphism
$M'\longrightarrow M'/M''$ by $\hat{V}_j$.
From $\overline{\gamma}_j^+\in V_j$
we, by Proposition \ref{prop:bassoc}, conclude that
$\Soc(\res_{Y_n'}^{\mathfrak{A}_n}(E))\leq V_1+V_2$,
so that $\res_{Y_n'}^{\mathfrak{A}_n}(E)=V_1+V_2$.
Hence it suffices to show that $V_1=V_2$.

\medskip
Assume that $V_1\neq V_2$. 
We first conclude from Lemma \ref{lemma:l3nl2} that 
$\res_{Q_n}^{\mathfrak{A}_n}(E)=V_1\oplus V_2$, where $\dim(V_j)=n_j$. 
As in the proof of Lemma \ref{lemma:l4nl2indec}, for $j=1,2,3$, we use the
$FP_n$-epimorphism $\pi_j:M\longrightarrow M_{n_j}$.
We show that $(\pi_1)_{|\hat{V}_1}$ is surjective:
We have $\hat{V}_2^{\pi_1}\leq\Rad(\res_{Y_{n_1}'}^{P_n}(M_{n_1}))$,
since otherwise taking
$w\in\hat{V}_2$ such that
$w^{\pi_1}\in M_{n_1}\setminus\Rad(\res_{Y_{n_1}'}^{P_n}(M_{n_1}))$
yields $w(y_{n_1}'+1)^{n_1-1}=a\gamma_1^+\in\hat{V}_2$, for some
$0\neq a\in F$,
a contradiction. From $\hat{V}_1+\hat{V}_{2}=M'$ we get
$(\hat{V}_1+\hat{V}_{2})^{\pi_1}=M_{n_1}$,
thus $\hat{V}_1^{\pi_1}=M_{n_1}$.
Hence from $\dim(\hat{V}_1)=n_1+1=\dim(M_{n_1})+1$ we
get $\dim(\hat{V}_1\cap\ker(\pi_1))=1$.
We set $v_0:=\gamma_2^++\gamma_3^+$. Then clearly $v_0\in\ker(\pi_1)$ 
and, since $v_0=\gamma^++\gamma_1^+$, we also
have $v_0\in\hat{V}_1$, thus 
$\hat{V}_1\cap\ker(\pi_1)=\langle v_0\rangle_F$.

\medskip
Now there is some $w\in \hat{V}_1$ such that $w^{\pi_1}=\gamma_1=\delta_1$. 
We write $w=\delta_1+\sum_{i=n_1+1}^na_i\delta_i$ 
with appropriate $a_{n_1+1},\ldots,a_n\in F$. Then we have
\begin{align*}
\tilde{w}:=w(y_{n_2}'+1)=&\sum_{i=n_1+2}^{n_1+n_2}(a_i+a_{i-1})\delta_i
                         +(a_{n_1+n_2}+a_{n_1+1})\delta_{n_1+1}\\
&+(a_{n-1}+a_n)(\delta_{n-1}+\delta_n)
\in\hat{V}_1\cap\ker(\pi_1).
\end{align*}
Since $w\in M'$, we have $\tilde{w}\neq 0$,
hence we get $\tilde{w}=av_0$, for some $0\neq a\in F$,
where we may suppose that $a=1$.
Thus we have $a_{i+1}=a_i+1$ for all $i\in\{n_1+1,\ldots,n-3\}$ 
and $a_n=a_{n-1}+1$.
Hence adding a suitable multiple of $v_0$ we may assume that
$ w=\delta_1+(b+1)\delta_{n-1}+b\delta_n
    +\sum_{i=0}^{n_2/2-1}\delta_{n_1+1+2i} $,
    for some $b\in F$.
Since $\hat{V}_1$ is an $FQ_n$-module, we also have 
$$ w w_{n_2,2}
= \delta_1+(b+1)\delta_{n-1}+b\delta_n
   +\sum_{i=1}^{n_2/2}\delta_{n_1+2i}
=w+\gamma_2^+ \in\hat{V}_1 ,$$
implying $\gamma_2^+\in\hat{V}_1\cap\ker(\pi_1)$, a contradiction. 
Hence we have $V_1=V_2$.

\medskip
It remains to show that $Q_n$ is a vertex of $E$. 
We follow the strategy given in Remark \ref{rem:strategy},
and assume that $E$ is relatively $R$-projective, for some maximal subgroup
$R<Q_n$.
By Lemma \ref{lemma:l3nl2} we have $\res_{Y_n'}^{Q_n}(V_1)=U_1\oplus U_2$ 
where the $U_i$ are indecomposable with vertex $Y_n'$. 
Since $\res_{Q_n}^{\mathfrak{A}_n}(E)=V_1$ is indecomposable, 
we infer that $Y_n'\leq R$. 
Thus, again by Lemma \ref{lemma:l3nl2}, 
$\res_R^{Q_n}(V_1)$ is either indecomposable or the direct sum
of an indecomposable module of dimension $n_1$ 
and an indecomposable module of dimension $n_2\neq n_1$,
a contradiction.
\end{proof}


\section{The case $n_l=2$ and $l=2$}\label{sec:nl2l2}
In order to complete the proof of Theorem \ref{thm:main},
we are now left with the case where
$l=2=n_l$ which is treated in the following.

\begin{prop}\label{prop:l2nl2}
Let $n_l=2$ and $l=2$. If $n>6$ then $E$ has vertex $Q_n$ and
source $\res_{Q_n}^{\mathfrak{A}_n}(E)$.
\end{prop}

\begin{proof}
We follow the strategy given in Remark \ref{rem:strategy}, and assume that 
$E$ is relatively $R$-projective for some maximal subgroup $R<Q_n$.
Hence we have $\langle \Phi(Q_n),Q_{n-4}\rangle\leq R$ and, in particular,
$X_{n_1}=X_n\leq\Phi(Q_n)\leq R$.
Using $\Phi(Q_n)=\Phi(P_n)=\Phi(P_{n_1})$ and 
$Q_{\frac{n_1}{2}}\leq Q_{n-4}$, 
Proposition \ref{prop:basis} yields
$B_{n_1}'=Q_{\frac{n_1}{2}}\Phi(P_{n_1})\leq Q_{n-4}\Phi(Q_n)\leq R$.
Since $Q_n/B_{n_1}'$ is elementary abelian of order 4, 
there are precisely three 
maximal subgroups of $Q_n$ containing $B_{n_1}'=B_n'$:
$$R_1':=B_n'\langle y_{n_1}'\rangle,\quad 
R_2':=B_n'\langle w_{n_1,1}\rangle=Q_{n_1},\quad 
R_3':=B_n'\langle w_{2,1}w_{2,2}\rangle=(B_n\times P_2)\cap Q_n.$$

\medskip
We next show that $E$ restricts indecomposably to $R_1'$ and $R_2'$:
By Remark \ref{rem:formerly51}, part (a), 
and using $\dim(E)=n_1=2\dim(FX_{n_1})$ we get 
$\res_{X_{n_1}}^{\mathfrak{A}_n}(E)\cong FX_{n_1}\oplus FX_{n_1}$.
Now let $i\in\{1,2\}$, fix an indecomposable direct sum decomposition of
$\res_{R_i'}^{\mathfrak{A}_n}(E)$, 
and assume that $\res_{R_i'}^{\mathfrak{A}_n}(E)$ is decomposable.
Hence $\res_{R_i'}^{\mathfrak{A}_n}(E)$ actually
has two indecomposable summands, 
both of which are not annihilated by $x_{n_1}^+$.
But, since $R_i'^{p_1 q_1}=\langle w_{n_1,1}\rangle$,
Remark \ref{rem:formerly51}, part (b),
implies that there is precisely one such summand, a contradiction.
Hence from this we also conclude that $\res_{Q_n}^{\mathfrak{A}_n}(E)$ 
is indecomposable. 

\medskip
It remains to show that $E$ is not relatively $R_3'$-projective. 
We claim that it suffices to show that $E$ restricts indecomposably to 
$H':=(\mathfrak{S}_{\frac{n_1}{2}}\times\mathfrak{S}_{\frac{n_1}{2}}
\times\mathfrak{S}_2)\cap\mathfrak{A}_n$.
Namely, then $E$ also restricts indecomposably to 
$\hat{H}':=((\mathfrak{S}_{\frac{n_1}{2}}\wr\mathfrak{S}_{2})
\times\mathfrak{S}_2)\cap\mathfrak{A}_n$. 
Of course $E$ is relatively $\hat{H}'$-projective, hence $E$ and 
$\res_{\hat{H}'}^{\mathfrak{A}_n}(E)$ have a common vertex. 
If $E$ were relatively $R_3'$-projective then 
$\res_{\hat{H}'}^{\mathfrak{A}_n}(E)$ 
would be relatively $H'$-projective, a contradiction since $|\hat{H}':H'|=2$.

\medskip
Therefore, we now show that $\End_{FH'}(E)$ is a local $F$-algebra so that 
$\res_{H'}^{\mathfrak{A}_n}(E)$ is indecomposable. 
Using the $F$-basis 
$\{\overline{\gamma}_1+\overline{\gamma}_n,\ldots,
\overline{\gamma}_{n_1}+\overline{\gamma}_n\}$ of $D$, 
we get
$$ \res_{\mathfrak{S}_{\frac{n_1}{2}}\times
\mathfrak{S}_{\frac{n_1}{2}}}^{\mathfrak{S}_n}(D)=M_1\oplus M_2\cong 
(M^{(\frac{n_1}{2}-1,1)}\boxtimes F)\oplus 
(F\boxtimes M^{(\frac{n_1}{2}-1,1)}) .$$
Here $M^{(\frac{n_1}{2}-1,1)}\boxtimes F$ denotes the
outer tensor product of the $F\mathfrak{S}_{\frac{n_1}{2}}$-modules 
$M^{(\frac{n_1}{2}-1,1)}$ and $F$.
Furthermore, $M_1$ has $F$-basis 
$\{\overline{\gamma}_1+\overline{\gamma}_n,\ldots,
\overline{\gamma}_{\frac{n_1}{2}}+\overline{\gamma}_n\}$,
and $M_2$ has $F$-basis 
$\{\overline{\gamma}_{\frac{n_1}{2}+1}+\overline{\gamma}_n,\ldots,
\overline{\gamma}_{n_1}+\overline{\gamma}_n\}$. 
Both modules are indecomposable and uniserial with 
descending composition factors
$$ (F,D^{(\frac{n_1}{2}-1,1)}\boxtimes F,F)
\quad\text{ and }\quad
(F,F\boxtimes D^{(\frac{n_1}{2}-1,1)},F),$$ 
respectively.
Note that this also holds for $n=6$, that is $n_1=4$, 
if we just let $D^{(1^2)}:=\{0\}$.
Now, since $n>6$, we have $n_1\geq 8$. Consequently
$\res_{\mathfrak{A}_{\frac{n_1}{2}}
 \times\mathfrak{A}_{\frac{n_1}{2}}}^{\mathfrak{S}_n}(D)
=\tilde{M}_1\oplus \tilde{M}_2$,
where $\tilde{M}_1$ and $\tilde{M}_2$, for $n_1>8$, are uniserial with 
descending composition factors
$$ (F,E_0^{(\frac{n_1}{2}-1,1)}\boxtimes F,F)
\quad\text{ and }\quad
(F,F\boxtimes E_0^{(\frac{n_1}{2}-1,1)},F) ,$$
respectively. 
For $n_1=8$ we have 
$\Soc(\tilde{M}_1)\cong\tilde{M}_1/\Rad(\tilde{M}_1)\cong F$
and
$\Rad(\tilde{M}_1)/\Soc(\tilde{M}_1)\cong
(E_+^{(3,1)}\oplus E_-^{(3,1)})\boxtimes F$
as well as
$\Soc(\tilde{M}_2)\cong\tilde{M}_2/\Rad(\tilde{M}_2)\cong F$
and 
$\Rad(\tilde{M}_2)/\Soc(\tilde{M}_2)\cong
F\boxtimes(E_+^{(3,1)}\oplus E_-^{(3,1)})$.
Thus
$$ \mathcal{E}:=
\End_{F[\mathfrak{S}_{\frac{n_1}{2}}\times\mathfrak{S}_{\frac{n_1}{2}}]}(D)=
\End_{F[\mathfrak{A}_{\frac{n_1}{2}}\times\mathfrak{A}_{\frac{n_1}{2}}]}(D)$$
has dimension 6, and $F$-basis $\{\varphi_1,\ldots,\varphi_6\}$ where 
$\varphi_1$ and $\varphi_2$ are the projections
onto $\tilde{M}_1$ and $\tilde{M}_2$, respectively, 
$\varphi_3$ annihilates $\tilde{M}_2$ and maps
$\tilde{M}_1$ onto $\Soc(\tilde{M}_1)$, 
$\varphi_4$ annihilates $\tilde{M}_1$ and maps
$\tilde{M}_2$ onto $\Soc(\tilde{M}_2)$, 
$\varphi_5$ annihilates $\tilde{M}_2$ and maps
$\tilde{M}_1$ onto $\Soc(\tilde{M}_2)$, 
$\varphi_6$ annihilates $\tilde{M}_1$ and maps
$\tilde{M}_2$ onto $\Soc(\tilde{M}_1)$. 
That is, for $a_1,\ldots,a_{n_1}\in F$ and
$v:=\sum_{i=1}^{n_1}a_i(\overline{\gamma}_i+\overline{\gamma}_n)\in D$, we
may suppose that
\begin{align*}
v^{\varphi_1}&= a_1(\overline{\gamma}_1+\overline{\gamma}_n)+\cdots 
+a_{\frac{n_1}{2}}(\overline{\gamma}_{\frac{n_1}{2}}+\overline{\gamma}_n),\\
v^{\varphi_2}&=a_{\frac{n_1}{2}+1}(\overline{\gamma}_{\frac{n_1}{2}+1}+
\overline{\gamma}_n)+\cdots +a_{n_1}(\overline{\gamma}_{n_1}+
\overline{\gamma}_n),\\
v^{\varphi_3}&=(a_1+\cdots+a_{\frac{n_1}{2}})((\overline{\gamma}_1+
\overline{\gamma}_n)+\cdots+(\overline{\gamma}_{\frac{n_1}{2}}+
\overline{\gamma}_n)),\\
v^{\varphi_4}&=(a_{\frac{n_1}{2}+1}+\cdots +a_{n_1})
((\overline{\gamma}_{\frac{n_1}{2}+1}+\overline{\gamma}_n)+\cdots+
(\overline{\gamma}_{n_1}+\overline{\gamma}_n)),\\
v^{\varphi_5}&=(a_1+\cdots+a_{\frac{n_1}{2}})
((\overline{\gamma}_{\frac{n_1}{2}+1}+\overline{\gamma}_n)+\cdots+
(\overline{\gamma}_{n_1}+\overline{\gamma}_n)),\\
v^{\varphi_6}&=(a_{\frac{n_1}{2}+1}+\cdots +a_{n_1})
((\overline{\gamma}_1+\overline{\gamma}_n)+\cdots+
(\overline{\gamma}_{\frac{n_1}{2}}+\overline{\gamma}_n)).
\end{align*}
The multiplication in $\mathcal{E}$ is given by 
\begin{align*}
\varphi_1^2=\varphi_1, & &
\varphi_1\varphi_3=\varphi_3=\varphi_3\varphi_1, & &
\varphi_6\varphi_1=\varphi_6=\varphi_2\varphi_6, \\
\varphi_2^2=\varphi_2, & &
\varphi_2\varphi_4=\varphi_4=\varphi_4\varphi_2, & &
\varphi_5\varphi_2=\varphi_5=\varphi_1\varphi_5, 
\end{align*}
and any other product of two basis elements vanishes. Note that
$H:=\mathfrak{S}_{\frac{n_1}{2}}\times\mathfrak{S}_{\frac{n_1}{2}}
\times\mathfrak{S}_2 
=(\mathfrak{S}_{\frac{n_1}{2}}\times 
\mathfrak{S}_{\frac{n_1}{2}})\langle (n-1,n)\rangle$ 
and 
$H'=(\mathfrak{A}_{\frac{n_1}{2}}\times 
\mathfrak{A}_{\frac{n_1}{2}})
\langle (1,2)(\frac{n_1}{2}+1,\frac{n_1}{2}+2),(1,2)(n-1,n)\rangle$.
The algebra $\mathcal{E}$ carries an $H$-algebra structure with 
respect to the conjugation action. We thus deduce
that $\tilde{\mathcal{E}}:=\End_{FH'}(E)=\End_{FH}(D)$, and 
$\tilde{\mathcal{E}}$ consists of those $\varphi\in\mathcal{E}$ which are
fixed under $(n-1,n)$. Let now $v\in D$ be as above, and let 
$\varphi:=b_1\varphi_1+\cdots+b_6\varphi_6\in\mathcal{E}$
for some $b_1,\ldots,b_6\in F$. Then 
$$v(n-1,n)=\sum_{i=1}^{n_1}(\sum_{j\neq i}a_j)
(\overline{\gamma}_i+\overline{\gamma}_n)=v^{\varphi_1+\cdots+\varphi_6}.$$
Hence $\varphi\in\tilde{\mathcal{E}}$ if and only if 
$\varphi(\varphi_1+\cdots+\varphi_6)=(\varphi_1+\cdots+\varphi_6)\varphi$, 
or equivalently, if $b_1=b_2$.
We have thus shown that $\tilde{\mathcal{E}}$ has dimension 5 and $F$-basis 
$\{\varphi_1+\varphi_2,\varphi_3,\varphi_4,\varphi_5,\varphi_6\}$. 
Since $\tilde{\mathcal{E}}$ is abelian and
$\varphi_3,\varphi_4,\varphi_5,\varphi_6$ are
nilpotent, we also deduce that 
$\Rad(\tilde{\mathcal{E}})$ has dimension 4. 
In particular, $\tilde{\mathcal{E}}$ is local, and the assertion follows.
\end{proof}

\begin{rem}\label{rem:n6}\normalfont
It remains to consider the case $n=6$. Let $E:=E^{(5,1)}_0$. 
In view of the above observations, we aim to show that $E$ has vertex 
$Q:=(\mathfrak{S}_2\times\mathfrak{S}_2\times\mathfrak{S}_2)\cap\mathfrak{A}_6
=\langle (1,2)(3,4),(3,4)(5,6)\rangle\cong V_4$
and sources of dimension 2:
To this end, let 
$Q<Q_6:=\langle (1,3)(2,4),(1,2)(3,4),(3,4)(5,6)\rangle<\mathfrak{A}_6$,
where $Q_6$ is a dihedral group of order $8$.
With respect to the basis
$\{\bar{\gamma}_1+\bar{\gamma}_6,\bar{\gamma}_2+\bar{\gamma}_6,
\bar{\gamma}_3+\bar{\gamma}_6,\bar{\gamma}_4+\bar{\gamma}_6\}$,
the action of $Q_6$ on $E$ is given via the following matrices:
$$
(1,3)(2,4)\longleftrightarrow 
\begin{pmatrix} .&.&1&.\\ .&.&.&1\\1&.&.&.\\.&1&.&. \end{pmatrix},$$
$$ (1,2)(3,4)\longleftrightarrow 
\begin{pmatrix} .&1&.&.\\1&.&.&.\\ .&.&.&1\\.&.&1&. \end{pmatrix},\quad
(3,4)(5,6)\longleftrightarrow
\begin{pmatrix} .&1&1&1\\ 1&.&1&1\\ 1&1&1&.\\ 1&1&.&1 \end{pmatrix}.$$
Let $\omega\in F$ be a primitive third root of unity, 
and consider the following $F$-subspaces of $E$:
\begin{align*}
U&:=\langle (\bar{\gamma}_1+\bar{\gamma}_6)
           +\omega(\bar{\gamma}_4+\bar{\gamma}_6),
            (\bar{\gamma}_2+\bar{\gamma}_6)
           +\omega(\bar{\gamma}_3+\bar{\gamma}_6)\rangle_F,\\
V&:=\langle (\bar{\gamma}_3+\bar{\gamma}_6)
           +\omega(\bar{\gamma}_2+\bar{\gamma}_6),
            (\bar{\gamma}_4+\bar{\gamma}_6)
           +\omega(\bar{\gamma}_1+\bar{\gamma}_6)\rangle_F.
\end{align*}
Then we have $E=U\oplus V$ as $F$-vector spaces,
and the action of $Q_6$ on $E$, with respect to
this new $F$-basis, is given via:
$$ (1,3)(2,4)\longleftrightarrow
\begin{pmatrix} .&.&1&.\\ .&.&.&1\\1&.&.&.\\.&1&.&. \end{pmatrix},$$
$$ (1,2)(3,4)\longleftrightarrow 
\begin{pmatrix} .&1&.&.\\1&.&.&.\\ .&.&.&1\\.&.&1&. \end{pmatrix},\quad
(3,4)(5,6)\longleftrightarrow
\begin{pmatrix} \omega&\omega^2&.&.\\ \omega^2&\omega&.&.\\ 
                .&.&\omega^2&\omega\\ .&.&\omega&\omega^2 \end{pmatrix}.$$
Thus both $U$ and $V$ are $FQ$-submodules of $E$, and
we have $\res_Q^{\mathfrak{A}_6}(E)=U\oplus V$,
where the socle of $U$ has dimension 1 
so that $U$ is indecomposable.
Moreover, from $((1,2)(3,4))^{(1,3)(2,4)}=(1,2)(3,4)$ and 
$$((3,4)(5,6))^{(1,3)(2,4)}=(1,2)(5,6)=(1,2)(3,4)\cdot (3,4)(5,6)$$
we infer that $\ind_Q^{Q_6}(U)\cong\res_{Q_6}^{\mathfrak{A}_6}(E)$.
Thus, by Green's Indecomposability Theorem,
$\res_{Q_6}^{\mathfrak{A}_6}(E)$ is indecomposable, and 
$E$ is relatively $Q$-projective. 
Moreover, each proper subgroup of $Q$ is cyclic, and
since $E$ belongs to a block of $F\mathfrak{A}_6$ with non-cyclic 
defect groups, $E$ cannot have a cyclic vertex, by \cite{E}. 
So $Q$ and $U$ then have to be vertex and source, respectively, of $E$.
\hfill\openbox
\end{rem}

\section{Appendix}\label{sec:appendix}
\normalfont

We give a new corrected proof of \cite[Thm. 1.4(a), 1.5(a)]{MuZ}.
For the case $l=2$ and $n_l=2$ we reuse the observations made 
in the proof of Proposition \ref{prop:l2nl2}, where we actually
have already chosen notation appropriately.

\begin{thm}\label{thm:correction}
Let $n$ be even, but not a $2$-power. Then 
$D$ has vertex $P_n$ and source $\res_{P_n}^{\mathfrak{S}_n}(D)$.
\end{thm}

\begin{proof}
By \cite[Prop. 3.7]{MuZ}, $\res_{P_n}^{\mathfrak{S}_n}(D)$ is indecomposable.
We again follow the strategy given in Remark \ref{rem:strategy}. Assume that 
$R<P_n$ is a maximal subgroup such that $D$ is relatively $R$-projective,
and fix an indecomposable direct sum decomposition of 
$\res_{R}^{\mathfrak{S}_n}(D)$. We have $X_n\leq\Phi(P_n)\leq R$.
Moreover, we have 
$\res_{\mathfrak{S}_{n-1}}^{\mathfrak{S}_n}(D)\cong D^{(n-2,1)}$.
The latter, by \cite[Thm. 1.2, 1.3]{MuZ}, has
vertex $P_{n-4}$ and trivial source. Thus we have $P_{n-4}\leq R$.

\medskip
Let first $n_l>2$, or $n_l=2$ and $l\geq 3$; recall that $l\geq 2$ anyway.
Then, in all these cases we have $P_{n_1}\leq P_{n-4}\leq R$,
implying that $R^{p_1q_1}=\langle w_{n_1,1}\rangle$. 
Then by Remark \ref{rem:formerly51}, part (b),
there is precisely
one indecomposable direct summand $U$ in the decomposition
which is not annihilated by $x_{n_1}^+$.
Thus $(FX_{n_1}\oplus FX_{n_1})|\res_{X_{n_1}}^{R}(U)$ implies
$\dim(U)>\dim(D)/2$, a contradiction.

\medskip
Let now $n_l=2$ and $l=2$.
Then using $\Phi(P_n)=\Phi(P_{n_1})$ and
$P_{\frac{n_1}{2}}\leq P_{n-4}$,
Proposition \ref{prop:basis} yields
$B_n=B_{n_1}=P_{\frac{n_1}{2}}\Phi(P_{n_1})\leq P_{n-4}\Phi(P_n)\leq R$.
Since $P_n/B_n$ is elementary abelian of order 4, there are precisely three
maximal subgroups of $P_n$ containing $B_n$, these are
$R_1:=B_n\langle y_{n_1}'\rangle$ and
$R_2:=B_n\langle y_{n_1}\rangle=P_{n_1}$ and 
$R_3:=B_n\langle w_{2,2}\rangle=B_n\times P_2$.
With the notation as in the proof of
Proposition \ref{prop:l2nl2}, we have $R_i\cap {\mathfrak{A}_n}=R_i'$ for $i\in\{1,2,3\}$.
Hence from the proof of Proposition \ref{prop:l2nl2}
we infer that $D$ restricts indecomposably to $R_1$ and $R_2$.


\medskip
It remains to show that $D$ is not relatively $R_3$-projective.
As in the proof of Proposition \ref{prop:l2nl2}, it suffices 
to show that $D$ restricts indecomposably to 
$H:=\mathfrak{S}_{\frac{n_1}{2}}\times\mathfrak{S}_{\frac{n_1}{2}}
\times\mathfrak{S}_2$.
Namely, then $D$ also restricts indecomposably to 
$\hat{H}:=(\mathfrak{S}_{\frac{n_1}{2}}\wr\mathfrak{S}_{2})
           \times\mathfrak{S}_2$. 
Of course $D$ is relatively $\hat{H}$-projective, hence $D$ and 
$\res_{\hat{H}}^{\mathfrak{S}_n}(D)$ have a common vertex. 
If $D$ were relatively $R_3$-projective then 
$\res_{\hat{H}}^{\mathfrak{S}_n}(D)$ 
would be relatively $H$-projective, a contradiction since $|\hat{H}:H|=2$.
But now the proof of Proposition \ref{prop:l2nl2} shows that
$\tilde{\mathcal{E}}=\End_{FH}(D)$ is a local $F$-algebra so that 
$\res_{H}^{\mathfrak{S}_n}(D)$ indeed is indecomposable. 
\end{proof}


\medskip
{\sc S.D.:
Mathematical Institute, University of Oxford \\
24-29 St Giles', Oxford, OX1 3LB, UK} \\
{\sf danz@maths.ox.ac.uk}

\medskip
{\sc J.M.:
Lehrstuhl D f\"ur Mathematik, RWTH Aachen \\
Templergraben 64, D-52062 Aachen, Germany} \\
{\sf Juergen.Mueller@math.rwth-aachen.de}

\end{document}